\documentclass[a4paper,reqno]{amsart}

\usepackage[T1]{fontenc}
\usepackage{lmodern}
\usepackage[a4paper,margin=30mm]{geometry}
\usepackage{microtype}
\usepackage{amsmath,amssymb,amsthm,mathtools}
\usepackage{enumitem}
\usepackage{booktabs,array,longtable}
\usepackage{xcolor}
\usepackage{xurl}
\usepackage[colorlinks=true,linkcolor=blue!52!black,citecolor=blue!52!black,
           urlcolor=blue!58!black]{hyperref}
\usepackage{comment}

\numberwithin{equation}{section}

\makeatletter
\@ifundefined{subjclassname@2020}{%
  \@namedef{subjclassname@2020}{2020 Mathematics Subject Classification}%
}{}
\makeatother

\allowdisplaybreaks
\setlist[itemize]{topsep=5pt,itemsep=3pt,parsep=2pt}
\setlist[enumerate]{topsep=5pt,itemsep=3pt,parsep=2pt}

\definecolor{navy}{RGB}{24,55,92}
\definecolor{wine}{RGB}{125,37,50}
\definecolor{forest}{RGB}{25,95,66}

\newtheorem{theorem}{Theorem}[section]
\newtheorem{lemma}[theorem]{Lemma}
\newtheorem{proposition}[theorem]{Proposition}
\newtheorem{corollary}[theorem]{Corollary}

\theoremstyle{definition}
\newtheorem{definition}[theorem]{Definition}

\theoremstyle{remark}
\newtheorem{remark}[theorem]{Remark}

\newcommand{\C}{\mathbb C}

\newcommand{\Cc}{C_c}
\newcommand{\Cr}{C_r^*}

\newcommand{\supp}{\operatorname{supp}}

\newcommand{\vN}{\mathbin{\overline{\otimes}}}

\begin{document}

\title[Von Neumann algebras as reduced twisted groupoid algebras]{Von Neumann algebras as reduced \\ twisted groupoid $C^*$-algebras}

\author[Alcides Buss]{Alcides Buss}
\address[Alcides Buss]{Departamento de Matem\'atica, Universidade Federal de Santa Catarina, 88.040-900
Florian\'opolis-SC, Brazil}
	\email{alcides.buss@ufsc.br}
	\urladdr{http://mtm.ufsc.br/~alcides/}

\author[Luiz F. Garcia]{Luiz Felipe Garcia}
\address[Luiz F. Garcia]{Departamento de Matem\'atica, Universidade Federal de Santa Catarina, 88.040-900
Florian\'opolis-SC, Brazil}
	\email{lfgarcia98@gmail.com}

\author[Tomás Pacheco]{Tomás Pacheco}
\address[Tomás Pacheco]{Center for Mathematical Analysis, Geometry and Dynamical Systems,
Department of Mathematics, Instituto Superior T\'ecnico, University of Lisbon,
Av. Rovisco Pais 1, 1049-001 Lisboa, Portugal}
	\email{tomas.pacheco@tecnico.ulisboa.pt}

\keywords{Von Neumann algebras, twisted groupoid $C^*$-algebras,
\'etale groupoids, Haar systems, subhomogeneous $C^*$-algebras,
controlled propagation, Calkin algebras}

\subjclass[2020]{Primary 46L10; Secondary 46L05, 22A22}

\date{\today}

\begin{abstract}
We characterize the von Neumann algebras that are isomorphic, as
$C^*$-algebras, to reduced twisted $C^*$-algebras of locally compact
Hausdorff groupoids equipped with continuous Haar systems of full
support. They are exactly the subhomogeneous von Neumann algebras,
equivalently finite products of matrix algebras over abelian von
Neumann algebras. The groupoid can always be chosen compact,
principal, and \'etale, with trivial twist and counting Haar system.

We also prove that, for a groupoid in this class, the full
or reduced twisted algebra is unital if and only if the groupoid
is \'etale with compact unit space. This criterion reduces the
classification for general Haar systems to the \'etale case.

A further obstruction comes from controlled propagation, defined
through faithful representations into the norm closure of uniformly
sparse matrices. Every reduced twisted \'etale groupoid algebra and
its Borel completion have controlled propagation, including for
non-Hausdorff groupoids with locally compact Hausdorff unit space.
For every infinite-dimensional Hilbert space $H$, every
$*$-homomorphism from a nonzero quotient of $B(H)$ into an algebra
with controlled propagation is zero. In particular, neither $B(H)$
nor the Calkin algebra embeds into any of these reduced or Borel
groupoid algebras. We also prove that every von Neumann algebra with
controlled propagation is finite. No separability or countability
assumptions are required.
\end{abstract}

\maketitle

\tableofcontents

\section{Introduction}
\label{sec:introduction}

We determine which von Neumann algebras occur as reduced twisted
$C^*$-algebras of locally compact Hausdorff groupoids equipped with
continuous Haar systems of full support. We first establish the
classification for \'etale groupoids, using counting measures.
An appendix then proves a unitality criterion which reduces the
general Haar-system case to this setting.

The question concerns the norm-completed groupoid $C^*$-algebra
itself. When this algebra has a von Neumann algebra structure,
its canonical regular representations need not be normal.
In the \'etale case, we do not assume that the diagonal
$C_0(G^{(0)})$ is ultraweakly closed.

Groupoid models connect the structure of a $C^*$-algebra with
topological and dynamical data; see Renault's foundational
monograph \cite{Renault1980}. The reconstruction results of
Kumjian \cite{Kumjian1986} and Renault \cite{Renault2008}
make this connection particularly explicit through diagonal
and Cartan subalgebras. In particular, every separable
$C^*$-algebra with a Cartan subalgebra admits a reduced twisted
\'etale groupoid model.

There are also obstructions to such realizations. Buss and
Sims \cite{BussSims2021} showed that untwisted groupoid
$C^*$-algebras are isomorphic to their opposite algebras,
yielding examples of $C^*$-algebras with no untwisted groupoid
model. Allowing twists changes the realization problem
substantially.

In \cite{BGP2026}, we proved that $B(H)$, for an
infinite-dimensional Hilbert space $H$, is not the reduced twisted
$C^*$-algebra of any locally compact Hausdorff \'etale groupoid.
Here we determine all von Neumann algebras that admit such a
realization. We also establish propagation obstructions that apply
to embeddings and homomorphisms, extend to non-Hausdorff groupoids,
and include all nonzero quotients of $B(H)$.

Recall that a $C^*$-algebra is \emph{subhomogeneous} if there is
an integer $N\geq1$ such that every irreducible representation
has dimension at most $N$ (see \cite[IV.1.4]{Blackadar2006}).

\begin{theorem}[Main theorem]
\label{thm:main}
Let $M$ be a von Neumann algebra. The following conditions are
equivalent.
\begin{enumerate}
\item There are a locally compact Hausdorff \'etale groupoid $G$
      and a twist $\Sigma$ over $G$ such that, as $C^*$-algebras,
      \[
              M\cong C_r^*(G,\Sigma).
      \]

\item The $C^*$-algebra underlying $M$ is subhomogeneous.

\item There are an integer $N\geq1$ and abelian von Neumann
      algebras $A_1,\ldots,A_N$, some of which may be zero, such
      that
      \[
              M\cong
              \prod_{n=1}^{N}
              \bigl(A_n\vN M_n(\C)\bigr).
      \]
\end{enumerate}
The product in \textup{(3)} is finite, and $\vN$ denotes the
von Neumann algebra tensor product. When these conditions hold,
the groupoid in \textup{(1)} may be chosen compact and principal,
and the twist may be chosen trivial.
\end{theorem}

The abelian factors in \textup{(3)} are arbitrary; in particular,
they may have diffuse parts. The restriction is the uniform
bound on the matrix degrees. For example, the finite type~I
von Neumann algebra
\[
        \prod_{n\geq1}M_n(\C)
\]
does not admit a groupoid realization as in \textup{(1)}, because
its coordinate representations have unbounded dimensions.

The groupoid model for \textup{(3)} is explicit. Write
$A_n\cong C(X_n)$ for each nonzero factor, where $X_n$ is its
compact Hausdorff Gelfand spectrum. The finite disjoint union
of the groupoids
\[
        X_n\times\{1,\ldots,n\}\times\{1,\ldots,n\},
\]
with the pair-groupoid operations in the finite coordinates,
has reduced $C^*$-algebra $M$. Section~\ref{sec:classification}
verifies this identification and completes the proof of the
theorem.

The classification extends to locally compact Hausdorff groupoids
with continuous Haar systems of full support.
The key observation is the following unitality criterion, proved
in Appendix~\ref{sec:unitality-haar}. Let $G$ be a nonempty locally
compact Hausdorff groupoid with such a Haar system $\lambda$,
and let $\Sigma$ be a twist over $G$. Then
\[
\begin{aligned}
 C_r^*(G,\Sigma;\lambda)\text{ is unital}
 &\Longleftrightarrow
 C^*(G,\Sigma;\lambda)\text{ is unital}\\
 &\Longleftrightarrow
 G\text{ is \'etale and }G^{(0)}\text{ is compact}.
\end{aligned}
\]
Here the notation records the Haar system used to define the
convolution and the completions.

Since every nonzero von Neumann algebra is unital, any realization
in this broader setting necessarily has an \'etale groupoid with
compact unit space. Remark~\ref{rem:haar-counting-normalization}
then replaces the given Haar system by counting measures without
changing the reduced algebra up to isomorphism.
Consequently, Corollary~\ref{cor:haar-system-classification}
extends the classification of Theorem~\ref{thm:main} to all locally
compact Hausdorff groupoids equipped with continuous Haar systems
of full support.

A second set of results concerns a propagation property of
$C^*$-algebras. For a set $I$, let $S(I)\subseteq B(\ell^2(I))$
be the norm closure of the operators whose matrices have a
uniform finite bound on the number of nonzero entries in each
row and each column. The bound may depend on the operator.
A $C^*$-algebra $A$ has \emph{controlled propagation} if it admits
an injective $*$-homomorphism
\[
        \rho\colon A\longrightarrow B(\ell^2(I))
        \qquad\text{with}\qquad
        \rho(A)\subseteq S(I)
\]
for some set $I$. Otherwise, it has \emph{uncontrolled
propagation}. Thus controlled propagation is a property of the
abstract $C^*$-algebra, formulated through the existence of a
suitable faithful representation. It passes to $C^*$-subalgebras.

In Section~\ref{sec:standard-inputs}, we also define
$B_r^*(G,\Sigma)$, the reduced completion of the bounded Borel
sections of the associated Fell line bundle that vanish outside
a compact subset of $G$. The propagation arguments apply to this
algebra as well as to $C_r^*(G,\Sigma)$.

\begin{theorem}[Propagation obstructions]
\label{thm:propagation-summary}
The following statements hold.
\begin{enumerate}
\item Let $(G,\Sigma)$ be a twisted locally compact \'etale
      groupoid with locally compact Hausdorff unit space.
      The arrow space need not be Hausdorff. Both
      \[
              C_r^*(G,\Sigma)
              \qquad\text{and}\qquad
              B_r^*(G,\Sigma)
      \]
      have controlled propagation.

\item Let $H$ be an infinite-dimensional Hilbert space and let
      $J\subsetneq B(H)$ be a norm-closed two-sided ideal.
      Then $B(H)/J$ has uncontrolled propagation.
      Moreover, for every $C^*$-algebra $A$ with controlled
      propagation, every $*$-homomorphism
      \[
              B(H)/J\longrightarrow A
      \]
      is zero.

\item Every von Neumann algebra with controlled propagation
      is finite.
\end{enumerate}
No separability assumption is imposed on $H$, and the
homomorphisms in \textup{(2)} are not required to be unital.
\end{theorem}

The first assertion follows from the finite-bisection structure
of compactly supported sections. In the source-regular
representations, a section supported in one bisection contributes
at most one nonzero matrix entry to each row and each column.
Finite sums and norm completion give controlled propagation,
uniformly over all source fibres. This argument uses the supports
of the sections and applies equally to the Borel completion.

The obstruction for $B(H)$ is proved in
Section~\ref{sec:sparse-obstruction}. In every faithful
representation of $B(H)$, we construct an operator at distance
one from $S(I)$. An amplification argument then shows that every
nonzero quotient of $B(H)$ contains a copy of $B(H)$. This gives
uncontrolled propagation for all such quotients and the
vanishing of homomorphisms in
Theorem~\ref{thm:propagation-summary}\textup{(2)}.
Finally, every nonfinite von Neumann algebra contains a
possibly nonunital copy of $B(\ell^2(\mathbb N))$, which proves
\textup{(3)}.

Taking $J=0$ or $J=\mathcal K(H)$ gives the following consequences
for $B(H)$ and the Calkin algebra
\[
        \mathcal Q(H)=B(H)/\mathcal K(H).
\]
Every $*$-homomorphism from either algebra into
$C_r^*(G,\Sigma)$ or $B_r^*(G,\Sigma)$ is zero. In particular,
neither algebra embeds as a $C^*$-subalgebra of a reduced or
Borel twisted \'etale groupoid algebra. These conclusions hold
in arbitrary infinite Hilbert-space dimension and allow
non-Hausdorff arrow spaces.

For full twisted \'etale groupoid algebras, the regular quotient
\[
        \lambda_{\mathrm{red}}\colon
        C^*(G,\Sigma)\longrightarrow C_r^*(G,\Sigma)
\]
shows that every $*$-homomorphism from $B(H)/J$ into
$C^*(G,\Sigma)$ has image contained in
$\ker\lambda_{\mathrm{red}}$. For nonempty $G$, this is a proper
ideal, so no such homomorphism is surjective. Thus no nonzero
quotient of $B(H)$, including $B(H)$ and $\mathcal Q(H)$, is
isomorphic to a full twisted \'etale groupoid $C^*$-algebra.

The unitality criterion in
Theorem~\ref{thm:unitality-haar} extends this non-isomorphism
conclusion to both full and reduced twisted algebras of locally
compact Hausdorff groupoids with continuous Haar systems of full
support. Indeed, an isomorphism with the unital algebra $B(H)/J$
would force the groupoid to be \'etale. After the normalization
in Remark~\ref{rem:haar-counting-normalization}, the regular
quotient would then give a nonzero homomorphism from $B(H)/J$
into an algebra with controlled propagation, a contradiction.

To complete the classification, we use the groupoid structure
to exclude the remaining finite von Neumann algebras.
A common input is the Banach-space quotient-rigidity lemma
proved in Section~\ref{sec:finite-dimensional-degree}: a corner
of a matrix algebra over a reduced twisted discrete group
algebra is finite-dimensional whenever it is a bounded linear
quotient of a von Neumann algebra. The quotient map need not
be multiplicative or normal.

Section~\ref{sec:finite-dimensional-degree} applies this lemma
to finite-dimensional irreducible representations of
$C_r^*(G,\Sigma)$. Such a representation comes from a finite
orbit $O$ and an irreducible projective representation $\sigma$
of a finite isotropy group, with
\[
        \dim\pi=|O|\dim\sigma.
\]
Separate estimates rule out unbounded orbit sizes and
unbounded projective isotropy dimensions along orthogonal
finite homogeneous central summands. This gives the uniform
degree bound for the type~I case.

Section~\ref{sec:type-II-obstruction} excludes finite type~II
central summands. The bisection estimates first rule out relative
diffuseness in the commutative von Neumann algebra generated by
the diagonal and the center. A relative Maharam lemma then
provides a relative atom. The corresponding corner has a
faithful family of quotients onto reduced isotropy corners.
Quotient rigidity makes these targets finite-dimensional,
which is incompatible with a nonzero type~II corner.

Section~\ref{sec:classification} combines finiteness,
the exclusion of type~II summands, and the bounded type~I degree
to prove Theorem~\ref{thm:main}.
Appendix~\ref{sec:unitality-haar} establishes the unitality criterion
and extends the classification to locally compact Hausdorff
groupoids with continuous Haar systems of full support.
The propagation results hold under the hypotheses of
Theorem~\ref{thm:propagation-summary}, which also allow
non-Hausdorff \'etale groupoids.
No countability or second-countability assumptions are imposed.

\section{Standard inputs and notation}
\label{sec:standard-inputs}

\subsection{\'Etale groupoids, twists, and regular representations}

For background on groupoid $C^*$-algebras and Haar systems,
see \cite{Renault1980}. For twists and their associated
reduced algebras, see \cite[Section~4]{Renault2008};
the Fell-bundle description and its relation to \'etale
groupoids are developed in \cite{BussExel2012}.
We recall the constructions needed below and specify our
conventions explicitly.

We first fix the conventions used throughout the paper. All groupoids
considered below have locally compact Hausdorff unit space, including in
the explicitly stated non-Hausdorff extensions.

\begin{definition}[\'Etale groupoid and bisection]
\label{def:etale-groupoid-bisection}
A \emph{topological groupoid} $G$ is a small category in which every arrow
is invertible, equipped with a topology for which multiplication and
inversion are continuous. Its unit space, with the subspace topology, is
denoted by $G^{(0)}$. For $\gamma\in G$, the source and range of $\gamma$
are denoted by $s(\gamma)$ and $r(\gamma)$. Thus $\gamma$ is viewed as an
arrow from $s(\gamma)$ to $r(\gamma)$, and a product $\alpha\beta$ is
defined exactly when $s(\alpha)=r(\beta)$.

The groupoid is \emph{\'etale} if the source map
$s\colon G\to G^{(0)}$ is a local homeomorphism. Since
$r=s\circ(\,\cdot\,)^{-1}$, the range map is then also a local
homeomorphism. An open set $U\subseteq G$ is an \emph{open bisection}
if both restrictions
\[
   s|_U\colon U\longrightarrow s(U),
   \qquad
   r|_U\colon U\longrightarrow r(U)
\]
are homeomorphisms onto open subsets of $G^{(0)}$.
\end{definition}

Every arrow of an \'etale groupoid has a neighbourhood which is an open
bisection. Consequently, every compact subset of $G$ is covered by
finitely many open bisections. Since the unit space is Hausdorff, every
open bisection is Hausdorff as well. These observations are the source of
the finite-propagation estimates below.

\begin{definition}[Twist and associated line bundle]
\label{def:twist-line-bundle}
Let $G$ be a locally compact groupoid with locally compact
Hausdorff unit space. A \emph{twist over $G$} is a central extension of topological groupoids
\[
   G^{(0)}\times\mathbb T
   \xrightarrow{\ \iota\ }\Sigma
   \xrightarrow{\ p\ }G,
\]
with the following properties:
\begin{enumerate}
\item $p$ is the identity on unit spaces and makes $\Sigma$ into a
      locally trivial principal $\mathbb T$-bundle over $G$;
\item $p^{-1}(G^{(0)})=\iota(G^{(0)}\times\mathbb T)$;
\item the circle action is central, in the sense that
      \[
         \iota(r(\sigma),z)\sigma
         =\sigma\iota(s(\sigma),z)
         \qquad(\sigma\in\Sigma,\ z\in\mathbb T).
      \]
\end{enumerate}
We write $z\sigma=\iota(r(\sigma),z)\sigma$.

The associated Fell line bundle (see \cite{Kumjian1998,SSW20,Renault2008}) is denoted by $L\to G$. Explicitly,
\[
   L=(\Sigma\times\mathbb C)/\mathbb T,
   \qquad
   z\cdot(\sigma,a)=(z\sigma,\overline z a),
\]
and the class of $(\sigma,a)$ is written
$[\sigma,a]\in L_{p(\sigma)}$. Multiplication and involution are induced
by
\[
   [\sigma,a][\tau,b]=[\sigma\tau,ab],
   \qquad
   [\sigma,a]^*=[\sigma^{-1},\overline a],
\]
where the product is defined when $\sigma$ and $\tau$ are composable.
The centrality condition makes these formulas independent of the chosen
representatives.
\end{definition}

Throughout the remainder of this section, $G$ is assumed to be
locally compact and \'etale, with locally compact Hausdorff unit space. The definition of a twist itself does not require
\'etaleness and will also be used in
Appendix~\ref{sec:unitality-haar}.

When the twist is trivial, $L$ is the product line bundle
$G\times\mathbb C$, and all formulas below reduce to the usual untwisted
formulas. Working with $L$ avoids choosing local scalar cocycles.

For a section $f$ of $L$, write
\[
   \operatorname{supp}^{\circ}(f)
   =\{\gamma\in G:f(\gamma)\ne0\}.
\]
We distinguish this set from its closure, the closed support of $f$.
In statements allowing non-Hausdorff $G$, the condition that $f$ be
\emph{compactly supported} will mean that it vanishes outside some compact
subset of $G$, or equivalently that
$\operatorname{supp}^{\circ}(f)$ is contained in such a subset. For
Hausdorff $G$, this agrees with compactness of the closed support.

If $G$ is Hausdorff, let $C_c(G,L)$ be the vector space of compactly
supported continuous sections of $L$. If $G$ is not Hausdorff,
$C_c(G,L)$ denotes the linear span of continuous sections on open
Hausdorff subsets whose supports are compact in those subsets, extended
by zero to $G$.

In either case, every $f\in C_c(G,L)$ is a finite sum
\[
   f=f_1+\cdots+f_m,
\]
where each $f_j$ is the extension by zero of a compactly supported
continuous section on an open bisection $U_j$ which trivializes $L$.
This follows from a finite bisection cover and a partition of unity;
in the non-Hausdorff case, apply this argument to each of the local
sections in the definition of $C_c(G,L)$. In particular, every element
of $C_c(G,L)$ vanishes outside a compact subset of $G$.

The convolution and involution formulas are
\begin{align*}
   (f*g)(\gamma)
      &=\sum_{\alpha\beta=\gamma}f(\alpha)g(\beta),\\
   f^*(\gamma)
      &=f(\gamma^{-1})^*.
\end{align*}
The sums are finite: the sections vanish outside compact sets, while
the source and range fibres are closed and discrete. Products and
inverses of open bisections are open bisections, so these operations
preserve $C_c(G,L)$ and make it a $*$-algebra.

For $u\in G^{(0)}$, put
\[
   G_u=s^{-1}(u),
   \qquad
   \mathcal H_u=\ell^2(G_u,L)
      :=\bigoplus_{\gamma\in G_u}L_\gamma.
\]
Thus a vector $\xi\in\mathcal H_u$ is a family
$(\xi_\gamma)_{\gamma\in G_u}$, with $\xi_\gamma\in L_\gamma$ and
\[
   \sum_{\gamma\in G_u}\|\xi_\gamma\|^2<\infty.
\]
The \emph{source-regular representation at $u$} is
\[
   \lambda_u\colon C_c(G,L)\longrightarrow B(\mathcal H_u),
   \qquad
   (\lambda_u(f)\xi)_\gamma
      =\sum_{\eta\in G_u}f(\gamma\eta^{-1})\xi_\eta.
\]
The product in each summand is the line-bundle multiplication
$L_{\gamma\eta^{-1}}L_\eta\to L_\gamma$.

If $f_j$ vanishes outside an open bisection, then $\lambda_u(f_j)$
is a weighted partial permutation of the one-dimensional summands of
$\mathcal H_u$, and
\[
   \|\lambda_u(f_j)\|\leq\|f_j\|_\infty.
\]
Applying a finite bisection decomposition therefore gives
\[
   \|\lambda_u(f)\|
   \leq\sum_{j=1}^m\|f_j\|_\infty,
\]
uniformly in $u$. The convolution and involution identities show that
each $\lambda_u$ is a $*$-representation.

\begin{definition}[Reduced twisted groupoid algebra]
\label{def:reduced-twisted-algebra}
The \emph{reduced norm} is
\[
   \|f\|_r=\sup_{u\in G^{(0)}}\|\lambda_u(f)\|.
\]
The completion of $C_c(G,L)$ for this norm is the
\emph{reduced twisted groupoid $C^*$-algebra}, denoted by
$C_r^*(G,\Sigma)$.
\end{definition}

To see that the regular representations separate sections, let
$\delta_u\in\mathcal H_u$ be the vector supported at the unit $u$,
with value $1\in L_u\cong\mathbb C$. For every $\gamma\in G$,
\[
   \bigl(\lambda_{s(\gamma)}(f)\delta_{s(\gamma)}\bigr)_\gamma
   =f(\gamma).
\]
Consequently, $\|f\|_\infty\leq\|f\|_r$.

Equivalently, if
\[
   \Lambda=\bigoplus_{u\in G^{(0)}}\lambda_u,
\]
then $C_r^*(G,\Sigma)$ is the norm closure of $\Lambda(C_c(G,L))$,
and $\Lambda$ extends to a faithful representation of
$C_r^*(G,\Sigma)$.

The restriction of the twist to the unit space is canonically trivial.
Extending sections on the unit space by zero therefore gives a canonical
copy
\[
   D=C_0(G^{(0)})\subseteq C_r^*(G,\Sigma),
\]
called the \emph{diagonal}, whose interplay with the groupoid structure is a central theme in the theory of Cartan subalgebras \cite{Kumjian1986, Renault2008}.

For the remainder of this paragraph, assume that $G$ is Hausdorff.
Restriction to the unit space extends to a faithful conditional
expectation
\[
   E_0\colon C_r^*(G,\Sigma)\longrightarrow D.
\]
Indeed, restriction is contractive for the reduced norm, fixes $D$,
and is $D$-bimodular. Positivity follows from
\[
   E_0(f^**f)(u)
   =\sum_{\gamma\in G_u}\|f(\gamma)\|^2
   \qquad(f\in C_c(G,L)).
\]
For faithfulness, suppose that $E_0(a^*a)=0$. Then
\[
   \|\lambda_u(a)\delta_u\|^2=E_0(a^*a)(u)=0
   \qquad(u\in G^{(0)}).
\]
Given $\eta\in G_u$, choose a unit vector $v_\eta\in L_\eta$.
Right multiplication by $v_\eta$ defines a unitary
\[
   R_\eta\colon\mathcal H_{r(\eta)}\longrightarrow\mathcal H_u,
   \qquad
   (R_\eta\xi)_{\alpha\eta}=\xi_\alpha v_\eta.
\]
It intertwines the corresponding regular representations and sends
$\delta_{r(\eta)}$ to the vector $v_\eta$ in the summand $L_\eta$.
Thus $\lambda_u(a)$ vanishes on every summand $L_\eta$.
The faithfulness of $\Lambda$ gives $a=0$.

\subsection{Sparse matrices and controlled propagation}

Let $I$ be an arbitrary set and let $(\delta_i)_{i\in I}$ be the standard
orthonormal basis of $\ell^2(I)$. For $T\in B(\ell^2(I))$, write
\[
   T_{ij}=\langle T\delta_j,\delta_i\rangle.
\]

\begin{definition}[Uniformly sparse operators]
\label{def:sparse-operator-algebra}
For an integer $m\geq1$, let $S_m(I)$ be the set of all
$T\in B(\ell^2(I))$ such that
\[
   \#\{i\in I:T_{ij}\ne0\}\leq m
   \quad\text{for every }j\in I
\]
and
\[
   \#\{j\in I:T_{ij}\ne0\}\leq m
   \quad\text{for every }i\in I.
\]
In words, every column and every row of the matrix of $T$ has at most
$m$ nonzero entries. Define
\[
   S(I)=\overline{\bigcup_{m\geq1}S_m(I)}^{\|\cdot\|}.
\]
\end{definition}

The set $S_m(I)$ need not be a linear subspace. Nevertheless,
\[
   S_m(I)+S_n(I)\subseteq S_{m+n}(I),
   \qquad
   S_m(I)S_n(I)\subseteq S_{mn}(I),
   \qquad
   S_m(I)^*=S_m(I).
\]
It follows that $S(I)$ is a $C^*$-subalgebra of $B(\ell^2(I))$.
Conjugation by a diagonal unitary or by a permutation unitary preserves
each $S_m(I)$. Thus rephasing or reordering the chosen basis does not
change these classes of operators.

For $I=\mathbb N$, the algebra $S(I)$ is the algebra
$B_f(\ell^2(\mathbb N))$ studied by Manuilov
\cite{Manuilov2019}. More generally, let
$\mathcal E_{\max}(I)$ consist of the subsets of $I\times I$
whose horizontal and vertical sections have uniformly
bounded finite cardinalities. This is the maximal uniformly
locally finite coarse structure on $I$; see
\cite[Section~2.3]{BragaFarahVignati2022}.
By the definition of the uniform Roe algebra,
\[
        S(I)=C_u^*(I,\mathcal E_{\max}(I)).
\]
For background on coarse structures and Roe algebras, see
\cite{Roe2003}. We use the explicit matrix description throughout.

\begin{definition}[Controlled propagation]
\label{def:controlled-propagation}
Let $A$ be a $C^*$-algebra.
\begin{enumerate}
\item A representation $\pi\colon A\to B(\ell^2(I))$ has
      \emph{controlled propagation} if $\pi(A)\subseteq S(I)$.
\item The algebra $A$ has \emph{controlled propagation} if it admits
      a faithful representation with controlled propagation.
\item The algebra $A$ has \emph{uncontrolled propagation} if no faithful
      representation of it has controlled propagation.
\end{enumerate}
\end{definition}

We should mention that controlled propagation is a property of abstract $C^*$-algebras, formulated through the existence of a suitable faithful representation. This concept is conceptually akin to the uniform sparsity conditions found in the theory of (uniform) Roe algebras and coarse geometry \cite{Roe2003}.

The definition is invariant under $C^*$-isomorphism. It requires the
existence of a faithful representation with the stated property;
other faithful representations may fail to have that property.

Controlled propagation passes to $C^*$-subalgebras by restricting a
faithful representation. Consequently, if a $C^*$-algebra $A$ contains
an injective copy of a $C^*$-algebra with uncontrolled propagation,
then $A$ also has uncontrolled propagation.

\begin{lemma}[Finite propagation]
\label{lem:finite-propagation}
Let $f\in C_c(G,L)$, and suppose that $f$ vanishes outside the union
of $m$ open bisections. After choosing unit vectors in the
one-dimensional fibres $L_\gamma$, the matrix of $\lambda_u(f)$
belongs to $S_m(G_u)$ for every $u\in G^{(0)}$.

Consequently, after making these choices simultaneously for all arrows,
the faithful direct sum of the source-regular representations satisfies
\[
   \Lambda\bigl(C_r^*(G,\Sigma)\bigr)\subseteq S(I_G),
   \qquad
   I_G:=\coprod_{u\in G^{(0)}}G_u.
\]
In particular, $C_r^*(G,\Sigma)$ has controlled propagation.
\end{lemma}

\begin{proof}
Fix $u\in G^{(0)}$. Choosing a unit vector in each line $L_\gamma$
identifies $\mathcal H_u$ with $\ell^2(G_u)$. The source-regular
formula shows that the matrix coefficient from the basis vector
indexed by $\eta\in G_u$ to the one indexed by $\gamma\in G_u$
can be nonzero only if
\[
   \gamma\eta^{-1}\in\operatorname{supp}^{\circ}(f).
   \tag{2.1}
\]

Let $U$ be one of the bisections in the chosen cover. If $\eta$ is fixed
and $\gamma\eta^{-1}\in U$, then
\[
   s(\gamma\eta^{-1})=r(\eta).
\]
Because $s|_U$ is injective, there is at most one element of $U$ with
this source. The map $\gamma\mapsto\gamma\eta^{-1}$ is injective,
so $U$ accounts for at most one nonzero entry in the column indexed
by $\eta$.

Similarly, if $\gamma$ is fixed and $\gamma\eta^{-1}\in U$, then
\[
   r(\gamma\eta^{-1})=r(\gamma).
\]
The injectivity of $r|_U$, together with the injectivity of
$\eta\mapsto\gamma\eta^{-1}$, shows that $U$ accounts for at most
one nonzero entry in the row indexed by $\gamma$.

Summing over the $m$ bisections proves that every row and every column
has at most $m$ nonzero entries. Hence
\[
   \lambda_u(f)\in S_m(G_u).
\]
The estimate is uniform in $u$. On identifying
$\bigoplus_u\mathcal H_u$ with $\ell^2(I_G)$, the block-diagonal
operator $\Lambda(f)$ therefore belongs to $S_m(I_G)$.

Every element of $C_c(G,L)$ vanishes outside a compact subset of $G$,
and every such compact subset has a finite open-bisection cover.
Thus
\[
   \Lambda(C_c(G,L))\subseteq\bigcup_{m\geq1}S_m(I_G).
\]
Taking norm closures yields the asserted inclusion for
$C_r^*(G,\Sigma)$. Since $\Lambda$ is faithful, it witnesses
controlled propagation.
\end{proof}

\subsection{Borel $C^*$-algebras of groupoids}

Continuity of the coefficients played no role in the preceding
row-and-column count. We now record the corresponding construction
for bounded Borel sections.

\begin{definition}[Compactly supported Borel algebra]
\label{def:compactly-supported-borel-algebra}
Let $B_c(G,L)$ be the space of Borel sections $f\colon G\to L$ such that
\[
   \|f\|_\infty:=\sup_{\gamma\in G}\|f(\gamma)\|<\infty
\]
and $f$ vanishes outside some compact subset $K\subseteq G$.
Here a Borel section is a Borel-measurable map satisfying
$f(\gamma)\in L_\gamma$ for every $\gamma\in G$.
\end{definition}

As specified above, the compactness condition concerns containment of
$\operatorname{supp}^{\circ}(f)$ in a compact set. This convention
applies whether or not $G$ is Hausdorff.

Every $f\in B_c(G,L)$ admits a finite bisection decomposition.
Indeed, choose a compact set $K$ outside which $f$ vanishes, and
cover $K$ by open bisections $U_1,\ldots,U_m$ which trivialize $L$.
Set
\[
   V_j=U_j\setminus\bigcup_{k<j}U_k,
   \qquad
   f_j=\mathbf1_{V_j}f
   \qquad(1\leq j\leq m),
\]
where multiplication by $\mathbf1_{V_j}$ is pointwise.
The sets $V_j$ are pairwise disjoint Borel bisections contained in
the respective $U_j$, and they cover $K$. Consequently,
\[
   f=f_1+\cdots+f_m,
   \qquad
   \operatorname{supp}^{\circ}(f_j)\subseteq K\cap V_j.
\]
Each $f_j$ is Borel, satisfies $\|f_j\|_\infty\leq\|f\|_\infty$,
and vanishes outside $K$, so $f_j\in B_c(G,L)$.
It also vanishes outside the open bisection $U_j$.
The sets $\operatorname{supp}^{\circ}(f_j)$ are pairwise disjoint;
no assertion about disjointness of their closures is needed.

The usual convolution and involution formulas make $B_c(G,L)$
a $*$-algebra:
\[
   (f*g)(\gamma)
      =\sum_{\alpha\beta=\gamma}f(\alpha)g(\beta),
   \qquad
   f^*(\gamma)=f(\gamma^{-1})^*.
\]
We verify the support and measurability assertions involved here.
Suppose that $f$ and $g$ vanish outside compact sets $K_f$ and $K_g$.
Then $f*g$ vanishes outside
\[
   K_fK_g
   =\{\alpha\beta:
       \alpha\in K_f,\ \beta\in K_g,\ s(\alpha)=r(\beta)\}.
\]
This set is compact. Indeed, the composable pairs form a closed subset
of $K_f\times K_g$, because $G^{(0)}$ is Hausdorff, and multiplication
is continuous.

The convolution sums are finite because the source and range fibres
are closed and discrete. If $K_f$ is covered by $m$ open bisections,
there are at most $m$ nonzero summands at any fixed $\gamma$, and hence
\[
   \|f*g\|_\infty
   \leq m\|f\|_\infty\|g\|_\infty.
\]
To check Borel measurability, decompose both sections using trivializing
open bisections. For two such pieces, vanishing outside $U$ and $V$,
respectively, their convolution vanishes outside $UV$. On $UV$ there
is a unique factorization with factors in $U$ and $V$, and those
factors depend continuously on the product. In the corresponding
line-bundle trivializations, the convolution is therefore a product
of Borel scalar functions and is Borel. Finite sums can likewise be
checked in finitely many trivializations covering the compact sets
outside which the sections vanish.

Inversion preserves Borel measurability and sends $K_f$ to the compact
set $K_f^{-1}$. Thus the involution also preserves $B_c(G,L)$.
Associativity and the $*$-identities follow from the corresponding
identities in the Fell line bundle and the finite convolution sums.

The source-regular formula extends to $B_c(G,L)$:
\[
   (\lambda_u(f)\xi)_\gamma
      =\sum_{\eta\in G_u}f(\gamma\eta^{-1})\xi_\eta.
\]
For a piece $f_j$ vanishing outside an open bisection, the resulting
operator is a weighted partial permutation and satisfies
\[
   \|\lambda_u(f_j)\|\leq\|f_j\|_\infty.
\]
A finite bisection decomposition therefore gives bounded operators
$\lambda_u(f)$, uniformly in $u$. The convolution identities show
that these operators define $*$-representations.

Set
\[
   \|f\|_r=\sup_{u\in G^{(0)}}\|\lambda_u(f)\|.
\]
For the decomposition constructed above,
\[
   \|f\|_r
   \leq\sum_{j=1}^m\|f_j\|_\infty
   \leq m\|f\|_\infty.
\]
The regular representations also separate Borel sections, since
\[
   \bigl(\lambda_{s(\gamma)}(f)
          \delta_{s(\gamma)}\bigr)_\gamma
   =f(\gamma).
\]
Consequently,
\[
   \|f\|_\infty\leq\|f\|_r.
\]
Thus $\|\cdot\|_r$ is a $C^*$-norm on $B_c(G,L)$.
Its completion is denoted by $B_r^*(G,\Sigma)$.

\begin{theorem}[Controlled propagation for the Borel completion]
\label{thm:controlled-propagation-borel}
Let $(G,\Sigma)$ be a twisted locally compact \'etale groupoid,
not necessarily Hausdorff. With
\[
   I_G=\coprod_{u\in G^{(0)}}G_u,
   \qquad
   \Lambda=\bigoplus_{u\in G^{(0)}}\lambda_u,
\]
one has
\[
   \Lambda\bigl(B_r^*(G,\Sigma)\bigr)\subseteq S(I_G).
\]
Consequently, $B_r^*(G,\Sigma)$ has controlled propagation.
The same is true of $C_r^*(G,\Sigma)$.
\end{theorem}

\begin{proof}
Choose unit vectors in all fibres $L_\gamma$, thereby identifying
the regular direct-sum Hilbert space with $\ell^2(I_G)$.

Let $f\in B_c(G,L)$ vanish outside one open bisection $U$.
The matrix coefficient of $\lambda_u(f)$ from $\eta$ to $\gamma$
can be nonzero only when $\gamma\eta^{-1}\in U$.
For fixed $\eta$, injectivity of $s|_U$ permits at most one such
$\gamma$; for fixed $\gamma$, injectivity of $r|_U$ permits at most
one such $\eta$. Hence
\[
   \lambda_u(f)\in S_1(G_u)
\]
for every $u$.

For an arbitrary $f\in B_c(G,L)$, use the preceding decomposition
\[
   f=f_1+\cdots+f_m,
\]
where each $f_j$ vanishes outside an open bisection. It follows that
\[
   \lambda_u(f)\in S_m(G_u)
   \quad\text{and}\quad
   \Lambda(f)\in S_m(I_G).
\]
Therefore
\[
   \Lambda(B_c(G,L))\subseteq\bigcup_{m\geq1}S_m(I_G).
\]
The representation $\Lambda$ is isometric for the defining reduced
norm. Taking norm closures gives
\[
   \Lambda\bigl(B_r^*(G,\Sigma)\bigr)\subseteq S(I_G),
\]
and faithfulness of $\Lambda$ gives controlled propagation.

Finally, every element of $C_c(G,L)$ is a finite sum of extensions
by zero of continuous sections with compact support on open Hausdorff
subsets. These extensions are Borel and vanish outside compact sets,
so
\[
   C_c(G,L)\subseteq B_c(G,L).
\]
The inclusion preserves convolution and involution and is isometric
for the common family of regular representations. It therefore
extends to an injective $*$-homomorphism
\[
   C_r^*(G,\Sigma)\longrightarrow B_r^*(G,\Sigma).
\]
Controlled propagation passes to $C^*$-subalgebras, giving the final
assertion.
\end{proof}

\section{Infinite von Neumann algebras}
\label{sec:sparse-obstruction}

We prove that $B(H)$ has uncontrolled propagation whenever $H$ is infinite-dimensional. The same is true of every nonzero quotient of $B(H)$, including the Calkin algebra.
These results give obstructions to embeddings into reduced and Borel
groupoid algebras, and to realizations as full groupoid algebras. We then
deduce that every von Neumann algebra with controlled propagation is finite.

Throughout this section, groupoids are locally compact and \'etale, with
locally compact Hausdorff unit space; the arrow space need not be
Hausdorff. Homomorphisms are not assumed to be unital unless this is
explicitly stated.

\subsection{The sparse obstruction for $B(H)$}

For the standard representation on $\ell^2(\mathbb N)$,
Manuilov constructs an isometry outside $S(\mathbb N)$
\cite[Proposition~2.1]{Manuilov2019}.
We prove below an obstruction that holds for every faithful
representation of $B(H)$, with no separability assumption
on either Hilbert space.

\begin{lemma}[Disjoint finite-support approximations]
\label{lem:disjoint-finite-localization}
Let $(\xi_n)_{n\ge1}$ be an orthonormal sequence in $\ell^2(I)$.
Given a finite set $F\subseteq I$, a number $\varepsilon>0$, and an
integer $n_0\ge1$, there are an index $n\ge n_0$ and a finitely
supported unit vector $\eta\in\ell^2(I)$ such that
\[
   \operatorname{supp}(\eta)\cap F=\varnothing,
   \qquad
   \|\xi_n-\eta\|<\varepsilon.
\]
Consequently, for every sequence $(\varepsilon_k)_{k\ge1}$ of positive
numbers, there are a subsequence $(\xi_{n_k})_{k\ge1}$ and finitely
supported unit vectors $(\eta_k)_{k\ge1}$ such that
\[
   \|\xi_{n_k}-\eta_k\|<\varepsilon_k
\]
and the sets $\operatorname{supp}(\eta_k)$ are pairwise disjoint.
\end{lemma}

\begin{proof}
By Bessel's inequality, an orthonormal sequence converges weakly to zero.
If $P_F$ is the coordinate projection onto $\ell^2(F)$, it follows that
\[
   \|P_F\xi_n\|^2
   =\sum_{i\in F}|\langle\xi_n,\delta_i\rangle|^2
   \longrightarrow0.
\]
Choose $0<\delta<\min\{1/2,\varepsilon/3\}$ and then $n\ge n_0$
with $\|P_F\xi_n\|<\delta$. Put $Q_F=1-P_F$. The vector
\[
   v=\frac{Q_F\xi_n}{\|Q_F\xi_n\|}
\]
is well-defined, has norm one, and is supported in $I\setminus F$.
Moreover,
\[
   \|\xi_n-v\|
   \le \|P_F\xi_n\|+\bigl|1-\|Q_F\xi_n\|\bigr|
   <2\delta,
\]
where $0\le1-\|Q_F\xi_n\|\le\|P_F\xi_n\|$ follows from the
reverse triangle inequality.

Truncating $v$ to a sufficiently large finite subset and normalizing gives
a finitely supported unit vector $\eta\in\ell^2(I\setminus F)$ with
$\|v-\eta\|<\delta$. Hence $\|\xi_n-\eta\|<3\delta<\varepsilon$.

For the final assertion, apply the first part inductively, taking $F$ to
be the union of the previously chosen supports and choosing each new
index larger than all the previous indices.
\end{proof}

\begin{proposition}[The sparse obstruction]
\label{prop:sparse-bh-obstruction}
Let $H$ be an infinite-dimensional Hilbert space, let $I$ be a set, and let
\[
   \rho\colon B(H)\longrightarrow B(\ell^2(I))
\]
be an injective $*$-homomorphism. There is a norm-one partial isometry
$T\in B(H)$ such that
\begin{equation}
   \operatorname{dist}\bigl(\rho(T),S_m(I)\bigr)=1
   \qquad(m\ge1).
   \tag{\thesection.1}
   \label{eq:sparse-bh-distance}
\end{equation}
In particular,
\[
   \operatorname{dist}\bigl(\rho(T),S(I)\bigr)=1,
   \qquad
   \rho(B(H))\not\subseteq S(I).
\]
\end{proposition}

\begin{proof}
\noindent\emph{Step 1: an orthonormal sequence in the representation space.}
Choose an orthonormal sequence $(h_j)_{j\ge1}$ in $H$, and define
rank-one matrix units by
\[
   e_{ij}\zeta=\langle\zeta,h_j\rangle h_i.
\]
Since $\rho$ is injective, $\rho(e_{11})\ne0$. Choose a unit vector
$\xi_1$ in its range and put
\[
   \xi_j=\rho(e_{j1})\xi_1\qquad(j\ge1).
\]
The matrix-unit identities give
\[
   \langle\xi_i,\xi_j\rangle
   =\langle\xi_1,\rho(e_{1i}e_{j1})\xi_1\rangle
   =\delta_{ij}.
\]
Thus $(\xi_j)$ is an orthonormal sequence in $\ell^2(I)$.

\smallskip
\noindent\emph{Step 2: choosing source and target vectors.}
Fix positive numbers $\varepsilon_k\to0$, for example
$\varepsilon_k=2^{-k}$. We recursively choose distinct indices
\[
   s_k,t_{k,1},\ldots,t_{k,N_k}
\]
and finitely supported unit vectors
\[
   \alpha_k,\beta_{k,1},\ldots,\beta_{k,N_k}
\]
such that
\begin{align}
   \|\xi_{s_k}-\alpha_k\|&<\varepsilon_k,
   \tag{\thesection.2}
   \label{eq:sparse-source-approx}\\
   \|\xi_{t_{k,j}}-\beta_{k,j}\|
      &<\frac{\varepsilon_k}{\sqrt{N_k}}
      \qquad(1\le j\le N_k),
   \tag{\thesection.3}
   \label{eq:sparse-target-approx}\\
   N_k&\ge k^2\,|\operatorname{supp}(\alpha_k)|.
   \tag{\thesection.4}
   \label{eq:sparse-target-count}
\end{align}
All supports chosen at all stages are required to be pairwise disjoint,
and no source or target index is reused.

To carry out stage $k$, apply
Lemma~\ref{lem:disjoint-finite-localization} beyond all indices already
used, avoiding the union of all supports already chosen. This gives
$s_k$ and $\alpha_k$. Once $\alpha_k$ is known, choose $N_k$ satisfying
\eqref{eq:sparse-target-count}. Apply the lemma another $N_k$ times with
tolerance $\varepsilon_k/\sqrt{N_k}$, each time avoiding all previously
chosen supports and moving farther along the sequence. This gives the
required target vectors.

\smallskip
\noindent\emph{Step 3: constructing a single operator.}
Put
\[
   v_k=\frac1{\sqrt{N_k}}
           \sum_{j=1}^{N_k}h_{t_{k,j}}.
\]
The vectors $(v_k)$ are orthonormal, as are the source vectors
$(h_{s_k})$. Define $T$ by
\[
   Th_{s_k}=v_k\quad(k\ge1),
   \qquad
   T=0\quad\text{on }
   \left(\overline{\operatorname{span}\{h_{s_k}:k\ge1\}}\right)^\perp.
\]
Then $T$ is a partial isometry of norm one.

Define also
\[
   \zeta_k=\frac1{\sqrt{N_k}}
               \sum_{j=1}^{N_k}\beta_{k,j}.
\]
Disjointness of the supports gives $\|\zeta_k\|=1$. The finite identity
\[
   Te_{s_k,1}
   =\frac1{\sqrt{N_k}}
       \sum_{j=1}^{N_k}e_{t_{k,j},1}
\]
implies
\[
   \rho(T)\xi_{s_k}
   =\frac1{\sqrt{N_k}}
       \sum_{j=1}^{N_k}\xi_{t_{k,j}}.
\]
This uses only multiplicativity of $\rho$; no normality is needed.
By \eqref{eq:sparse-target-approx},
\[
   \|\rho(T)\xi_{s_k}-\zeta_k\|
   \le\frac1{\sqrt{N_k}}
        \sum_{j=1}^{N_k}
          \|\xi_{t_{k,j}}-\beta_{k,j}\|
   <\varepsilon_k.
\]
Since $\|\rho(T)\|=1$, we obtain
\begin{equation}
   \|\rho(T)\alpha_k-\zeta_k\|
   \le\|\alpha_k-\xi_{s_k}\|
       +\|\rho(T)\xi_{s_k}-\zeta_k\|
   <2\varepsilon_k.
   \tag{\thesection.5}
   \label{eq:sparse-spread-approx}
\end{equation}

\smallskip
\noindent\emph{Step 4: estimating the distance from sparse operators.}
Fix $m\ge1$ and $b\in S_m(I)$, and let
\[
   L_k=\operatorname{supp}(b\alpha_k).
\]
Since each column of $b$ has at most $m$ nonzero entries,
\begin{equation}
   |L_k|\le m\,|\operatorname{supp}(\alpha_k)|.
   \tag{\thesection.6}
   \label{eq:sparse-column-count}
\end{equation}
Let $P_{L_k}$ be the coordinate projection onto $\ell^2(L_k)$.
At most $|L_k|$ of the pairwise disjoint supports of the vectors
$\beta_{k,j}$ can meet $L_k$, and each $\|P_{L_k}\beta_{k,j}\|$
is at most one. Therefore
\begin{equation}
\begin{split}
   \|P_{L_k}\zeta_k\|^2
   &=\frac1{N_k}\sum_{j=1}^{N_k}
                      \|P_{L_k}\beta_{k,j}\|^2\\
   &\le\frac{|L_k|}{N_k}
    \le\frac{m\,|\operatorname{supp}(\alpha_k)|}{N_k}
    \le\frac{m}{k^2}.
\end{split}
   \tag{\thesection.7}
   \label{eq:sparse-captured-mass}
\end{equation}

Now $\alpha_k$ is a unit vector and
$(1-P_{L_k})b\alpha_k=0$. For $k^2\ge m$, it follows that
\begin{equation}
\begin{split}
   \|\rho(T)-b\|
   &\ge\|(1-P_{L_k})(\rho(T)-b)\alpha_k\|\\
   &=\|(1-P_{L_k})\rho(T)\alpha_k\|\\
   &\ge\|(1-P_{L_k})\zeta_k\|
         -\|\rho(T)\alpha_k-\zeta_k\|\\
   &\ge\sqrt{1-\frac{m}{k^2}}-2\varepsilon_k.
\end{split}
   \tag{\thesection.8}
   \label{eq:sparse-distance-bound}
\end{equation}
Letting $k\to\infty$ with $m$ and $b$ fixed gives
$\|\rho(T)-b\|\ge1$. Since $b\in S_m(I)$ was arbitrary, and
$0\in S_m(I)$ while $\|\rho(T)\|=1$, this proves
\eqref{eq:sparse-bh-distance}. The same distance holds for the union
of the $S_m(I)$ and for its norm closure $S(I)$.
\end{proof}

\begin{corollary}
\label{cor:bh-uncontrolled}
If $H$ is infinite-dimensional, then $B(H)$ has uncontrolled propagation.
\end{corollary}

\begin{proof}
Proposition~\ref{prop:sparse-bh-obstruction} excludes every faithful
representation whose image is contained in some $S(I)$.
\end{proof}

\subsection{Quotients and groupoid algebras}

The next observation allows us to handle arbitrary Hilbert-space
dimension without using a classification of the ideals of $B(H)$.

\begin{proposition}[Copies of $B(H)$ in its nonzero quotients]
\label{prop:bh-in-nonzero-quotients}
Let $H$ be infinite-dimensional and let $J\subsetneq B(H)$ be a
norm-closed two-sided ideal. Then $B(H)/J$ contains a unital copy of
$B(H)$. Consequently, every nonzero quotient of $B(H)$ has
uncontrolled propagation.
\end{proposition}

\begin{proof}
For every infinite Hilbert-space dimension $\kappa$, one has
$\kappa\cdot\kappa=\kappa$. We may therefore choose a unitary
\[
   W\colon H\otimes H\longrightarrow H.
\]
Let $q_J\colon B(H)\to B(H)/J$ be the quotient map and define
\[
   \Theta(T)=W(T\otimes I_H)W^*,
   \qquad
   \Phi(T)=q_J(\Theta(T)).
\]
Both maps are unital $*$-homomorphisms. To prove that $\Phi$ is
injective, suppose that $T\ne0$ and choose a unit vector $\eta\in H$
with $T\eta\ne0$. The operator
\[
   V_\eta\colon H\longrightarrow H,
   \qquad V_\eta\xi=W(\eta\otimes\xi),
\]
is an isometry and satisfies
\[
   V_\eta^*\Theta(T)^*\Theta(T)V_\eta
   =\|T\eta\|^2 I_H.
\]
If $\Phi(T)=0$, then $\Theta(T)\in J$. The displayed identity and
the ideal property would imply $I_H\in J$, contrary to $J\ne B(H)$.
Thus $\Phi$ is injective.

By Corollary~\ref{cor:bh-uncontrolled}, the embedded copy of $B(H)$
has uncontrolled propagation. Since controlled propagation passes to
$C^*$-subalgebras, $B(H)/J$ has uncontrolled propagation as well.
\end{proof}

\begin{corollary}[Calkin algebras]
\label{cor:calkin-uncontrolled}
For every infinite-dimensional Hilbert space $H$, the Calkin algebra
\[
   \mathcal Q(H)=B(H)/\mathcal K(H)
\]
contains a unital copy of $B(H)$ and has uncontrolled propagation.
\end{corollary}

\begin{proof}
Apply Proposition~\ref{prop:bh-in-nonzero-quotients} to the proper ideal
$\mathcal K(H)$.
\end{proof}

\begin{remark}[Inclusions between Calkin algebras]
\label{rem:calkin-corner-inclusions}
Let $H_1$ be a closed subspace of $H_2$, let
$j\colon H_1\hookrightarrow H_2$ be the inclusion isometry, and let
$p=jj^*$. Extension by zero gives an injective $*$-homomorphism
\[
   B(H_1)\longrightarrow B(H_2),
   \qquad T\longmapsto jTj^*.
\]
Moreover, $jTj^*$ is compact if and only if $T$ is compact: one direction
follows by multiplication of a compact operator by bounded operators,
and the other follows from $T=j^*(jTj^*)j$.
Writing $q_i\colon B(H_i)\to\mathcal Q(H_i)$, we therefore obtain
an injective $*$-homomorphism
\[
   \mathcal Q(H_1)\longrightarrow\mathcal Q(H_2),
   \qquad q_1(T)\longmapsto q_2(jTj^*).
\]
Its range is the corner
$q_2(p)\mathcal Q(H_2)q_2(p)$, since every operator in $pB(H_2)p$
has the form $jTj^*$. In particular, choosing a separably
infinite-dimensional closed subspace of $H$ also reduces the uncontrolled
propagation of $\mathcal Q(H)$ to the separable case.
\end{remark}

\begin{corollary}[No nonzero maps into algebras with controlled propagation]
\label{cor:bh-quotients-no-controlled-maps}
Let $H$ be infinite-dimensional, let $J\subsetneq B(H)$ be a
norm-closed two-sided ideal, and let $A$ be a $C^*$-algebra with
controlled propagation. Every $*$-homomorphism
\[
   B(H)/J\longrightarrow A
\]
is zero. In particular, this holds with domain $B(H)$ or $\mathcal Q(H)$.
\end{corollary}

\begin{proof}
If such a homomorphism were nonzero, its image would be a nonzero quotient
of $B(H)$. By Proposition~\ref{prop:bh-in-nonzero-quotients}, that
image would have uncontrolled propagation. On the other hand, the image
is a $C^*$-subalgebra of $A$ and therefore has controlled propagation,
a contradiction.
\end{proof}

\begin{corollary}[Reduced and Borel groupoid algebras]
\label{cor:bh-calkin-no-groupoid-embedding}
Let $(G,\Sigma)$ be a twisted locally compact \'etale groupoid, not
necessarily Hausdorff, with locally compact Hausdorff unit space.
Let $H$ be any infinite-dimensional Hilbert space.
Every $*$-homomorphism from $B(H)$ or $\mathcal Q(H)$ into either
\[
   C_r^*(G,\Sigma)
   \qquad\text{or}\qquad
   B_r^*(G,\Sigma)
\]
is zero. In particular, neither $B(H)$ nor $\mathcal Q(H)$ embeds as a
$C^*$-subalgebra of either algebra, and neither is $*$-isomorphic to
either algebra. The same conclusions hold for every nonzero quotient
of $B(H)$.
\end{corollary}

\begin{proof}
The two target algebras have controlled propagation by
Theorem~\ref{thm:controlled-propagation-borel}. Apply
Corollary~\ref{cor:bh-quotients-no-controlled-maps}.
\end{proof}

For the full twisted groupoid algebra, we use the canonical surjection
\[
   \lambda_{\mathrm{red}}\colon C^*(G,\Sigma)
                 \longrightarrow C_r^*(G,\Sigma).
\]
It is induced by the identity on $C_c(G,L)$: the universal norm dominates
the reduced norm, and the resulting $*$-homomorphism has closed range
containing the dense convolution algebra. This construction also applies
when $G$ is not Hausdorff.

\begin{corollary}[Full groupoid algebras]
\label{cor:bh-quotients-not-full-groupoid}
Let $(G,\Sigma)$ be a nonempty twisted locally compact \'etale groupoid,
not necessarily Hausdorff, with locally compact Hausdorff unit space.
Let $H$ be infinite-dimensional and let $J\subsetneq B(H)$ be a
norm-closed two-sided ideal. Then
\[
   B(H)/J\not\cong C^*(G,\Sigma).
\]
In particular, neither $B(H)$ nor $\mathcal Q(H)$ is $*$-isomorphic to
$C^*(G,\Sigma)$.

More precisely, every $*$-homomorphism
\[
   \varphi\colon B(H)/J\longrightarrow C^*(G,\Sigma)
\]
satisfies
\[
   \varphi(B(H)/J)\subseteq\ker\lambda_{\mathrm{red}}.
\]
Consequently, if $C^*(G,\Sigma)$ is unital, there is no unital
$*$-homomorphism from $B(H)/J$ into $C^*(G,\Sigma)$.
\end{corollary}

\begin{proof}
The composition $\lambda_{\mathrm{red}}\circ\varphi$ is zero by
Corollary~\ref{cor:bh-calkin-no-groupoid-embedding}. This proves the
range inclusion.

Since $G$ is nonempty, $C_0(G^{(0)})\ne0$ and hence
$C_r^*(G,\Sigma)\ne0$. Thus $\ker\lambda_{\mathrm{red}}$ is a
proper ideal, so $\varphi$ cannot be surjective. If the full algebra is
unital, $\lambda_{\mathrm{red}}$ is unital as well, and a unital
$\varphi$ would give a nonzero unital composition, a contradiction.
\end{proof}

\begin{remark}
The reduced and Borel conclusions exclude embeddings even when they are
not unital. For the full algebra, the argument above shows that any
embedding would have to lie in $\ker\lambda_{\mathrm{red}}$; it does
not exclude such nonunital embeddings. The propagation arguments in this section use the \'etale
hypothesis. The non-isomorphism conclusions for full and reduced
algebras also extend to locally compact Hausdorff groupoids with
continuous Haar systems of full support, by
Theorem~\ref{thm:unitality-haar} and
Remark~\ref{rem:haar-counting-normalization}.
\end{remark}

\subsection{Nonfinite von Neumann algebras}

\begin{theorem}[Nonfinite von Neumann algebras have uncontrolled propagation]
\label{thm:nonfinite-uncontrolled}
Every nonfinite von Neumann algebra $M$ contains a possibly nonunital
copy of $B(\ell^2(\mathbb N))$. Consequently, $M$ has uncontrolled
propagation.
\end{theorem}

\begin{proof}
Represent $M$ faithfully and normally on a Hilbert space $\mathcal H$.
Since $M$ is not finite, there are projections $f<e$ in $M$ with
$e\sim f$. Choose a partial isometry $v\in M$ such that
\[
   v^*v=e,\qquad vv^*=f,
\]
and put $p_0=e-f\ne0$. Then $v\in eMe$ and $v^*p_0=0$.
The projections
\[
   p_n=v^np_0(v^*)^n\qquad(n\ge0)
\]
are nonzero and pairwise orthogonal. Indeed, $v$ is an isometry on
$e\mathcal H$, and for $\xi,\eta\in\mathcal H_0:=p_0\mathcal H$,
\[
   \langle v^n\xi,v^m\eta\rangle
   =\delta_{nm}\langle\xi,\eta\rangle.
\]
Here $v^0p_0(v^*)^0$ means $p_0$.

Using the standard basis $(\delta_n)_{n\ge0}$ of
$\ell^2(\mathbb N_0)$, define
\[
   U\colon\ell^2(\mathbb N_0)\otimes\mathcal H_0
                  \longrightarrow\mathcal H,
   \qquad U(\delta_n\otimes\xi)=v^n\xi.
\]
This is an isometry onto the closed orthogonal sum
$\bigoplus_{n\ge0}p_n\mathcal H$. Consequently,
\[
   \Phi(T)=U(T\otimes I_{\mathcal H_0})U^*
\]
defines an injective $*$-homomorphism
$B(\ell^2(\mathbb N_0))\to B(\mathcal H)$.

To see that its range lies in $M$, put
\[
   w_{ij}=v^ip_0(v^*)^j\in M\qquad(i,j\ge0).
\]
These are matrix units. If $P_N$ is the projection onto
$\operatorname{span}\{\delta_0,\ldots,\delta_{N-1}\}$ and
$T_{ij}$ are the matrix entries of $T$, then
\[
   \Phi(P_NTP_N)=\sum_{i,j=0}^{N-1}T_{ij}w_{ij}\in M.
\]
These operators are uniformly bounded by $\|T\|$ and converge strongly
to $\Phi(T)$. Strong closedness of $M$ therefore gives $\Phi(T)\in M$.
The unit of this copy is $\sum_{n\ge0}^{\mathrm{sot}}p_n$, which may
be smaller than $1_M$.

Corollary~\ref{cor:bh-uncontrolled} and the fact that controlled
propagation passes to $C^*$-subalgebras now imply that $M$ has
uncontrolled propagation.
\end{proof}

\begin{corollary}[Finiteness of groupoid von Neumann algebras]
\label{cor:no-properly-infinite-summand}
Every von Neumann algebra with controlled propagation is finite.
In particular, let $(G,\Sigma)$ be a twisted locally compact \'etale
groupoid, not necessarily Hausdorff, with locally compact Hausdorff
unit space. If either $C_r^*(G,\Sigma)$ or $B_r^*(G,\Sigma)$ is
$*$-isomorphic to a von Neumann algebra, that von Neumann algebra is
finite. In particular, it has no nonzero properly infinite central summand.
\end{corollary}

\begin{proof}
The first assertion is the contrapositive of
Theorem~\ref{thm:nonfinite-uncontrolled}. The remaining assertions follow
from Theorem~\ref{thm:controlled-propagation-borel} and the invariance of
controlled propagation under $C^*$-isomorphism.
\end{proof}

\section{Finite-dimensional representations and bounded type I degree}
\label{sec:finite-dimensional-degree}

Throughout this section, $G$ is a locally compact Hausdorff \'etale
groupoid, $\Sigma$ is a twist over $G$, and $L\to G$ is the associated
Fell line bundle. We assume that
\[
   M=\Cr(G,\Sigma)
\]
is a von Neumann algebra, with its given $C^*$-norm and multiplication.
By Corollary~\ref{cor:no-properly-infinite-summand}, $M$ is finite.

If $M$ is of type I, its homogeneous decomposition therefore has the form
\[
   M\cong\prod_{n\ge1}\bigl(Z_n\vN M_n(\mathbb C)\bigr),
\]
where the $Z_n$ are abelian von Neumann algebras, some of which may be
zero. Our task is to show that only finitely many of the $Z_n$ are
nonzero. Finiteness alone does not give this conclusion: the von Neumann
algebra $\prod_{n\ge1}M_n(\mathbb C)$ is finite and has unbounded
homogeneous degrees.

The additional information comes from the groupoid. We shall show that a
finite-dimensional irreducible representation has dimension
\[
   \dim\pi=|O|\dim\sigma,
\]
where $O$ is a finite orbit and $\sigma$ is a projective representation
of a finite isotropy group. We then bound the two factors separately.
The preliminary lemmas do not require $M$ to be of type I; the quotient
lemma will also be used in the next section.

\subsection{A rigidity lemma for Banach quotients of von Neumann algebras}

A Banach space $X$ is \emph{Grothendieck} if every weak-*
convergent sequence in $X^*$ is weakly convergent.
Every von Neumann algebra has this property by
\cite[Corollary~7]{Pfitzner1994}. For a shorter proof of
Pfitzner's underlying weak-compactness theorem, see
\cite{FernandezPoloPeralta2010}.

Let $\Gamma$ be a discrete group and let $\omega$ be a normalized
scalar two-cocycle. Our convention for the twisted regular unitaries is
\[
   \lambda_\omega(g)\delta_h
   =\omega(g,h)\delta_{gh}.
\]
Their closed linear span is denoted by $\Cr(\Gamma,\omega)$.

\begin{lemma}[Von Neumann quotient rigidity]
\label{lem:vn-quotient-rigidity}
Let $N$ be a von Neumann algebra, let $k\ge1$, and let
$q\in M_k(\Cr(\Gamma,\omega))$ be a projection. If there is a
bounded surjective linear map
\[
   Q\colon N\longrightarrow qM_k(\Cr(\Gamma,\omega))q,
\]
then its target is finite-dimensional. The map $Q$ need not be
multiplicative or normal.
\end{lemma}

\begin{proof}
We first recall why a separable $C^*$-algebra which is a Banach quotient
of a von Neumann algebra must be finite-dimensional.

The Grothendieck property passes to bounded linear quotients. Indeed, if
$T\colon X\to Y$ is a bounded surjection, $T^*$ identifies $Y^*$
with a closed subspace of $X^*$. Weak-* convergence in $Y^*$ gives
weak-* convergence of the images in $X^*$. If $X$ is Grothendieck,
the latter convergence is weak. By Hahn--Banach, the weak topology of
a closed subspace is the topology inherited from the ambient space,
so the original sequence converges weakly in $Y^*$.

If $Y$ is also separable, its dual unit ball is weak-* compact and
metrizable. Every sequence in that ball therefore has a weak-* convergent
subsequence, which is weakly convergent when $Y$ is Grothendieck.
By the Eberlein--\v{S}mulian theorem, the dual unit ball is weakly compact.
Thus $Y^*$, and hence $Y$, is reflexive. Finally, a $C^*$-algebra is
reflexive only if it is finite-dimensional. To recall the obstruction,
every infinite-dimensional $C^*$-algebra contains an infinite-dimensional
abelian $C^*$-subalgebra; choosing norm-one positive functions with
pairwise disjoint supports in its spectrum gives a closed linear
subspace isometric to $c_0$, which is not reflexive.

Put $B=qM_k(\Cr(\Gamma,\omega))q$ and suppose that $B$ is
infinite-dimensional. Choose a linearly independent sequence
$(b_j)_{j\ge1}$ in $B$. Approximate every matrix entry of
$q,b_1,b_2,\ldots$ by sequences of finite linear combinations of the
twisted regular unitaries. The group elements used in these approximations
generate a countable subgroup $\Lambda\le\Gamma$, and
\[
   q,b_1,b_2,\ldots
   \in M_k\bigl(\Cr(\Lambda,\omega|_\Lambda)\bigr).
\]

Here the subgroup algebra is canonically a $C^*$-subalgebra of the
ambient reduced algebra. To see this, decompose $\ell^2(\Gamma)$ over
the cosets $\Lambda t$. For each coset, the unitary
\[
   \ell^2(\Lambda)\longrightarrow\ell^2(\Lambda t),
   \qquad \delta_h\longmapsto\omega(h,t)\delta_{ht},
\]
intertwines the twisted regular actions of $\Lambda$. Thus the algebraic
inclusion preserves the reduced norm.

Compression to $\ell^2(\Lambda)$ gives the conditional expectation
\[
   E_\Lambda\colon\Cr(\Gamma,\omega)\longrightarrow
       \Cr(\Lambda,\omega|_\Lambda),
   \qquad
   E_\Lambda(\lambda_\omega(g))=
   \begin{cases}
      \lambda_{\omega|_\Lambda}(g),&g\in\Lambda,\\
      0,&g\notin\Lambda.
   \end{cases}
\]
Indeed, this formula defines a completely positive contraction, fixes
the subgroup algebra, and is bimodular over it.

Apply this expectation entry by entry. Since $q$ belongs to the subgroup
matrix algebra, bimodularity shows that $E_\Lambda^{(k)}$ maps $B$
onto
\[
   B_\Lambda
   :=qM_k\bigl(\Cr(\Lambda,\omega|_\Lambda)\bigr)q.
\]
The algebra $B_\Lambda$ is separable because $\Lambda$ is countable,
and it is infinite-dimensional because it contains every $b_j$.
On the other hand,
\[
   E_\Lambda^{(k)}\circ Q\colon N\longrightarrow B_\Lambda
\]
is a bounded surjection. The preceding Banach-space argument forces
$B_\Lambda$ to be finite-dimensional, a contradiction.
\end{proof}

The use of Banach quotients is important here: the maps arising from
$M$ below need not be normal, and their kernels need not be weakly closed.

\subsection{Finite orbits and reduced factorization}

We first record how to prescribe operators in finite homogeneous
summands. We use the structure theorem for homogeneous von Neumann
algebras; see \cite[Chapter~V]{TakesakiI}.

\begin{lemma}[Constant-matrix lifts in homogeneous summands]
\label{lem:sec4-constant-matrix-lift}
Let $N$ be a homogeneous type $I_n$ von Neumann algebra, where
$n\ge1$, and let $\rho\colon N\to B(H)$ be irreducible.
Then $\dim H=n$, and every $V\in B(H)$ has a lift $x\in N$ with
\[
   \rho(x)=V,
   \qquad \|x\|=\|V\|.
\]
If $V$ is unitary, the lift may be chosen unitary.
\end{lemma}

\begin{proof}
Identify $N$ with $Z(N)\vN M_n(\mathbb C)$. Its center acts under
$\rho$ through a character $\chi$. Every element of $N$ is a finite
sum $\sum_{p,q}z_{pq}\otimes e_{pq}$. Consequently, the image of $N$
equals the image of its constant matrix algebra, so the restriction of
$\rho$ to that matrix algebra is irreducible.

Thus $\dim H=n$, and there is a unitary $U\colon\mathbb C^n\to H$
such that
\[
   \rho(z\otimes a)=\chi(z)UaU^*.
\]
The constant matrix $x=1\otimes U^*VU$ has the required image and norm.
It is unitary when $V$ is unitary.
\end{proof}

We also use the strong summability of elements supported on orthogonal
central projections. If $(z_i)_{i\in I}$ are pairwise orthogonal central
projections in $M$ and $x_i\in z_iM$ with
$C:=\sup_i\|x_i\|<\infty$, then
\[
   x=\sum_{i\in I}^{\mathrm{sot}}x_i\in M,
   \qquad z_ix=x_i,
   \qquad \|x\|=\sup_i\|x_i\|.
\]
Here the sum is the net of finite partial sums. Indeed, in a faithful
normal representation, for each finite $F\subseteq I$,
\[
   \left\|\sum_{i\in F}x_i\xi\right\|^2
   =\sum_{i\in F}\|x_iz_i\xi\|^2
   \le C^2\sum_{i\in F}\|z_i\xi\|^2.
\]
Strong summability of the orthogonal projections makes the tails tend
to zero; see \cite[Theorem~4.1.3]{Murphy}. Strong closedness gives the
limit in $M$, and multiplication by $z_i$ gives $z_ix=x_i$ and the
norm equality.

Combining these observations gives a construction used in both obstruction
arguments. Suppose that each $z_iM$ is homogeneous of finite degree and
that $\pi_i\colon M\to B(H_i)$ is irreducible with
$\pi_i(z_i)=I_{H_i}$. For any contractions $T_i\in B(H_i)$, we can
choose lifts $x_i\in z_iM$ and form their strong sum to obtain a
contraction $x\in M$ satisfying
\begin{equation}
   \pi_i(x)=\pi_i(z_ix)=\pi_i(x_i)=T_i
   \qquad(i\in I).
   \tag{\ref{lem:sec4-constant-matrix-lift}.1}
   \label{eq:sec4-prescribed-images}
\end{equation}
This uses only the algebraic identity $z_ix=x_i$; it does not require
the possibly singular representation $\pi_i$ to preserve a strong limit.

\begin{lemma}[Finite-orbit form]
\label{lem:finite-orbit-form}
Let $\pi\colon M\to B(H_\pi)$ be an irreducible finite-dimensional
representation. There are a finite orbit $O\subseteq G^{(0)}$, a point
$u\in O$, a finite isotropy group $\Gamma=G_u^u$, and an irreducible
projective representation $\sigma\colon\Gamma\to U(K)$ such that
$\pi$ factors through the reduced restriction map
\[
   R_O\colon M\longrightarrow\Cr(G|_O,\Sigma|_O)
\]
and is equivalent to the representation induced from $\sigma$. The
multiplier of $\sigma$ is determined, up to cohomology, by
$\Sigma|_\Gamma$. In particular,
\begin{equation}
   H_\pi\cong\ell^2(O)\otimes K,
   \qquad \dim H_\pi=|O|\dim K.
   \tag{\ref{lem:finite-orbit-form}.1}
   \label{eq:sec4-orbit-dimension}
\end{equation}
\end{lemma}

\begin{proof}
Put $D=C_0(G^{(0)})$.

\smallskip
\noindent\emph{Step 1: the diagonal determines a finite orbit.}
An approximate identity $(e_j)$ of $D$ is an approximate identity for
$M$. Indeed, for $f\in\Cc(G,L)$,
\[
   (e_j*f)(\gamma)=e_j(r(\gamma))f(\gamma),
   \qquad
   (f*e_j)(\gamma)=f(\gamma)e_j(s(\gamma)).
\]
These converge uniformly to $f$ on its compact support, and a finite
bisection decomposition gives convergence in reduced norm. Hence
$\pi|_D$ is nondegenerate. Simultaneous diagonalization gives a finite
set $F\subseteq G^{(0)}$ and nonzero orthogonal subspaces $H_x$ with
\begin{equation}
   H_\pi=\bigoplus_{x\in F}H_x,
   \qquad \pi(d)|_{H_x}=d(x)I_{H_x}\quad(d\in D).
   \tag{\ref{lem:finite-orbit-form}.2}
   \label{eq:sec4-diagonal-blocks}
\end{equation}

Let $\gamma\colon x\to y$ with $x\in F$. Choose a section
$f\in\Cc(U,L)$ on an open bisection $U\ni\gamma$, with
$c:=\|f(\gamma)\|>0$. For $\xi\in H_x$,
\[
   \|\pi(f)\xi\|^2
   =\langle\pi(f^**f)\xi,\xi\rangle
   =c^2\|\xi\|^2.
\]
For $d\in D$, the section $h=d*f-d(y)f$ is supported in $U$ and
vanishes at its unique arrow with source $x$. Thus $(h^**h)(x)=0$ and
\[
   \|(\pi(d)-d(y)I)\pi(f)\xi\|^2
   =\langle\pi(h^**h)\xi,\xi\rangle=0.
\]
Therefore $\pi(f)\xi$ is a joint eigenvector at $y$, and it is nonzero
when $\xi\ne0$. This forces $y\in F$ and
$\pi(f)H_x\subseteq H_y$. The same argument for $f^*$ gives the
reverse transport. In fact, if $T=\pi(f)|_{H_x}$, then
\[
   T^*T=c^2I_{H_x},
   \qquad TT^*=c^2I_{H_y},
\]
so $c^{-1}T$ is unitary.

It follows that $F$ is invariant and that the dimensions of the $H_x$
are constant on each orbit. Since bisection-supported sections span
$\Cc(G,L)$, the sum of the $H_x$ over any orbit in $F$ reduces
$\pi(M)$. Irreducibility forces $F$ to be a single orbit, denoted by $O$.

\smallskip
\noindent\emph{Step 2: algebraic restriction and reduced restriction.}
First note that
\begin{equation}
   f|_{G|_O}=0\quad\Longrightarrow\quad\pi(f)=0
   \qquad(f\in\Cc(G,L)).
   \tag{\ref{lem:finite-orbit-form}.3}
   \label{eq:sec4-algebraic-restriction-kernel}
\end{equation}
Indeed, choose a scalar partition of unity on $\supp(f)$ subordinate
to finitely many bisections and write $f=\sum_j f_j$, where
$f_j=\psi_j f$. Every $f_j$ still vanishes on $G|_O$. Since $O$ is
invariant, $(f_j^**f_j)(x)=0$ for every $x\in O$. The diagonal
decomposition therefore gives
$\pi(f_j)^*\pi(f_j)=\pi(f_j^**f_j)=0$, so $\pi(f_j)=0$.

We next construct restriction on the reduced completions. The finite
orbit $O$ is closed and invariant. For $u\in O$, its reduction has the
same source fibre as $G$:
\[
   (G|_O)_u=G_u.
\]
Identifying the corresponding regular representations gives
\[
   \|f|_{G|_O}\|_r
   =\sup_{u\in O}\|\lambda_u(f)\|
   \le\|f\|_r.
\]
Thus algebraic restriction, which respects convolution and involution,
extends to a contractive $*$-homomorphism
\[
   R_O\colon M\longrightarrow\Cr(G|_O,\Sigma|_O).
\]

This map is surjective. The reduction $G|_O$ is discrete: a bisection
through an arrow $\gamma$ can be shrunk so that its source meets $O$
only at $s(\gamma)$, and then it meets $G|_O$ only at $\gamma$.
Every compactly supported section on the reduction has finite support.
Extend each of its values using a local trivialization of $L$ and a
compactly supported cutoff on such a bisection, and sum these extensions.
Algebraic restriction is therefore onto. Its completed range is closed
and contains the dense convolution algebra, proving surjectivity of $R_O$.

\smallskip
\noindent\emph{Step 3: the transitive matrix decomposition.}
Write $m=|O|$, fix $u\in O$, and put $\Gamma=G_u^u$.
For each $x\in O$, choose an arrow $\gamma_x\colon u\to x$ and a
unit vector $v_x\in L_{\gamma_x}$, with $\gamma_u=u$ and $v_u=1_u$.
Also choose unit vectors $w_g\in L_g$ for $g\in\Gamma$, with
$w_e=1_u$, and define $\omega$ by
\[
   w_gw_h=\omega(g,h)w_{gh}.
\]
This is a normalized scalar two-cocycle representing the restricted twist.

Every arrow $\alpha\colon y\to x$ in $G|_O$ has a unique expression
$\alpha=\gamma_xg\gamma_y^{-1}$ with $g\in\Gamma$. For
$b\in\Cc(G|_O,L|_{G|_O})$, define scalar matrix entries by
\[
   b(\gamma_xg\gamma_y^{-1})
   =\Phi(b)_{xy}(g)\,v_xw_gv_y^*.
\]
The identities $v_y^*v_y=1_u$ and $w_aw_b=\omega(a,b)w_{ab}$ give
\[
   \Phi(b*c)_{xz}(g)
   =\sum_{y\in O}\ \sum_{g_1g_2=g}
       \Phi(b)_{xy}(g_1)\Phi(c)_{yz}(g_2)\omega(g_1,g_2).
\]
The involution likewise becomes the twisted matrix adjoint. Thus $\Phi$
is an algebraic $*$-isomorphism onto the matrix algebra over the twisted
group ring of $\Gamma$.

For completeness, this isomorphism preserves the reduced norm. For each
source point $y\in O$, define a unitary
\[
   W_y\colon\ell^2((G|_O)_y,L)
          \longrightarrow\ell^2(O)\otimes\ell^2(\Gamma)
\]
by sending the unit vector $v_xw_gv_y^*$ in the basis line
$L_{\gamma_xg\gamma_y^{-1}}$ to $\delta_x\otimes\delta_g$.
The same multiplication formula gives
\[
   W_y\lambda_y(b)W_y^*
   =\bigl[\lambda_\omega(\Phi(b)_{xz})\bigr]_{x,z\in O}.
\]
Taking the supremum over $y$ proves equality of the reduced norms.
Consequently, $\Phi$ extends to an isomorphism
\begin{equation}
   \overline\Phi\colon\Cr(G|_O,\Sigma|_O)
       \xrightarrow{\ \cong\ }M_m\bigl(\Cr(\Gamma,\omega)\bigr).
   \tag{\ref{lem:finite-orbit-form}.4}
   \label{eq:sec4-transitive-matrix-decomposition}
\end{equation}

\smallskip
\noindent\emph{Step 4: finiteness and factorization.}
The composition $\overline\Phi\circ R_O$ is a bounded surjection
from $M$ onto $M_m(\Cr(\Gamma,\omega))$.
Lemma~\ref{lem:vn-quotient-rigidity}, with $k=m$ and $q=1$, implies
that this matrix algebra is finite-dimensional. Hence so is
$\Cr(\Gamma,\omega)$. Since
\[
   \lambda_\omega(g)\delta_e=\delta_g,
\]
the twisted regular unitaries are linearly independent, and $\Gamma$
must be finite. The arrow parametrization above now shows that $G|_O$
is a finite groupoid.

For a section $b$ on this finite reduction, choose an extension
$\widetilde b\in\Cc(G,L)$ and set $\tau(b)=\pi(\widetilde b)$.
Equation~\eqref{eq:sec4-algebraic-restriction-kernel} makes this
well-defined, and algebraic restriction shows that $\tau$ is a
$*$-representation. Its domain is finite-dimensional and already equals
its reduced completion, so $\tau$ is bounded. Equality on $\Cc(G,L)$
and continuity now give
\[
   \pi=\tau\circ R_O.
\]
Surjectivity of $R_O$ implies that $\tau$ is irreducible.

An irreducible representation of $M_m(\Cr(\Gamma,\omega))$ is the
matrix amplification of an irreducible representation
$\rho\colon\Cr(\Gamma,\omega)\to B(K)$. This follows by restricting
to a diagonal matrix corner and using the matrix units to identify the
other corners. Put $\sigma(g)=\rho(\lambda_\omega(g))$. Then
$\sigma$ is an irreducible $\omega$-projective representation and,
under the resulting identification $H_\pi=\ell^2(O)\otimes K$,
\begin{equation}
   [\pi(f)]_{xy}
   =\sum_{g\in\Gamma}
      \Phi(f|_{G|_O})_{xy}(g)\,\sigma(g)
   \qquad(f\in\Cc(G,L)).
   \tag{\ref{lem:finite-orbit-form}.5}
   \label{eq:sec4-induced-block-formula}
\end{equation}
This is the matrix form of the representation induced from $\sigma$,
and it gives $\dim H_\pi=|O|\dim K$.
Changing the vectors $w_g$ changes $\omega$ by a coboundary, as required.
\end{proof}

\begin{remark}[Reduced factorization]
Equation~\eqref{eq:sec4-algebraic-restriction-kernel} concerns only the
convolution algebra. It does not, by itself, identify the completed
kernel of $R_O$. We first used quotient rigidity to prove that the
reduction is finite; only then did the algebraic representation become
automatically continuous. No exactness assumption on $G$ is used.
\end{remark}

\subsection{The orbit-growth obstruction}

Formula~\eqref{eq:sec4-induced-block-formula} has a useful consequence:
a bisection sends a fixed source block into at most one range block.
We compare this with an operator that spreads a vector from one source
block uniformly over the orbit. No bound on the dimensions of the blocks
is needed.

\begin{lemma}[Orbit-growth obstruction]
\label{lem:orbit-growth-obstruction}
Let $(z_i)_{i\ge1}$ be pairwise orthogonal nonzero central projections
in $M$, with each $z_iM$ homogeneous of finite degree. Let
$\pi_i\colon M\to B(H_i)$ be irreducible representations satisfying
$\pi_i(z_i)=I_{H_i}$, and write their finite-orbit realizations as
\[
   H_i=\ell^2(O_i)\otimes K_i.
\]
Then the orbit cardinalities $|O_i|$ are bounded.
\end{lemma}

\begin{proof}
Suppose otherwise and pass to a subsequence with $n_i:=|O_i|\to\infty$.
Choose $u_i\in O_i$ and put
\[
   \theta_i=\frac1{\sqrt{n_i}}\sum_{x\in O_i}\delta_x.
\]
Define a norm-one operator $T_i$ on $H_i$ by
\[
   T_i(\delta_{u_i}\otimes\eta)=\theta_i\otimes\eta
   \quad(\eta\in K_i),
   \qquad
   T_i=0\quad\text{on }(\mathbb C\delta_{u_i}\otimes K_i)^\perp.
\]
By \eqref{eq:sec4-prescribed-images}, there is a contraction $x\in M$
with $\pi_i(x)=T_i$ for every $i$.

Fix $f\in\Cc(G,L)$ and write $f=\sum_{j=1}^m f_j$ with each $f_j$
supported in an open bisection. The integer $m$ depends on $f$ but not
on $i$. Each bisection contains at most one arrow with source $u_i$.
By \eqref{eq:sec4-induced-block-formula}, there is a subset
$S_i\subseteq O_i$, with $|S_i|\le m$, such that
\[
   \pi_i(f)(\mathbb C\delta_{u_i}\otimes K_i)
   \subseteq\bigoplus_{y\in S_i}(\mathbb C\delta_y\otimes K_i).
\]
Let $Q_i$ be the projection onto the complementary orbit blocks and
choose a unit vector $\eta_i\in K_i$. For
$\xi_i=\delta_{u_i}\otimes\eta_i$, we have
$Q_i\pi_i(f)\xi_i=0$, whereas
\[
   \|Q_iT_i\xi_i\|^2=1-\frac{|S_i|}{n_i}.
\]
For all sufficiently large $i$,
\[
   \|x-f\|
   \ge\|Q_i\pi_i(x-f)\xi_i\|
   =\sqrt{1-\frac{|S_i|}{n_i}}
   \ge\sqrt{1-\frac{m}{n_i}}.
\]
Letting $i\to\infty$ with $f$ fixed gives $\|x-f\|\ge1$.
This holds for every $f\in\Cc(G,L)$, contradicting norm density.
\end{proof}

\subsection{The isotropy-degree obstruction}

The comparison between Jordan bounds and metric entropy
is closely related to the method of Breuillard and Pisier
\cite[Section~2]{BreuillardPisier2017}. Here we need a
version for linear combinations of a fixed number of
projective representation operators. We give the required
argument explicitly, using Collins' finite-group bound.

On an isotropy block, a bisection contributes at most a scalar multiple
of one projective group operator. We therefore need a unitary that stays
a fixed positive distance from all linear combinations of a prescribed
number of such operators.

The matrix argument below fixes that number first and then lets the
dimension grow. Its geometric input is a comparison of covering numbers:
the relevant linear combinations lie in a union of relatively small
linear spaces, while the unitary group requires many more balls to cover
it at the same scale. All norms in this subsection are operator norms.

\begin{lemma}[A covering bound for the unitary group]
\label{lem:sec4-unitary-entropy}
There is a numerical constant $\varepsilon_0>0$ such that $U(d)$
cannot be covered by fewer than $2^{d^2}$ balls of radius
$2\varepsilon_0$, for any $d\ge1$. The centers of the balls may be
arbitrary matrices in $M_d(\mathbb C)$.
\end{lemma}

\begin{proof}
The self-adjoint matrices form a real normed space of dimension $d^2$.
Let $B$ be its closed unit ball, and choose a maximal
$1/2$-separated subset $S\subseteq B$. This set is finite and its
radius-$1/2$ balls cover $B$. Comparing Lebesgue volumes gives
\[
   \operatorname{vol}(B)
   \le |S|\,2^{-d^2}\operatorname{vol}(B),
   \qquad |S|\ge2^{d^2}.
\]

For $X=X^*$, its Cayley transform
\[
   C(X)=(I+iX)(I-iX)^{-1}
\]
is unitary. If $X,Y\in B$, direct multiplication gives
\[
   (I-iX)\bigl(C(X)-C(Y)\bigr)(I-iY)=2i(X-Y).
\]
Since $\|I-iX\|,\|I-iY\|\le\sqrt2$, this implies
\[
   \|C(X)-C(Y)\|\ge\|X-Y\|.
\]
Thus the unitaries $C(X)$, for $X\in S$, are pairwise separated by
at least $1/2$. Take $\varepsilon_0=1/100$. A ball of radius
$2\varepsilon_0$ has diameter at most $4\varepsilon_0<1/2$, so it
contains at most one of these unitaries. The result follows.
\end{proof}

For a nonempty set $\Omega\subseteq M_d(\mathbb C)$ and an integer
$m\ge1$, write
\[
   \mathcal L_m(\Omega)
   =\left\{\sum_{j=1}^m c_jU_j:
          c_j\in\mathbb C,\ U_j\in\Omega\right\}.
\]
Zero coefficients and repetitions are allowed, so this is the set of
linear combinations of at most $m$ elements of $\Omega$.

\begin{lemma}[Sparse projective spans]
\label{lem:sec4-sparse-projective-spans}
Let $\varepsilon_0$ be as in
Lemma~\ref{lem:sec4-unitary-entropy}. For every integer $m\ge1$ there
exists an integer $d_0(m)$ with the following property. If $\Gamma$
is finite and $\sigma\colon\Gamma\to U(d)$ is a projective
representation with $d\ge d_0(m)$, then some $V\in U(d)$ satisfies
\[
   \operatorname{dist}\bigl(V,\mathcal L_m(\sigma(\Gamma))\bigr)
   \ge\varepsilon_0.
\]
The threshold $d_0(m)$ is independent of $\Gamma$ and $\sigma$.
\end{lemma}

\begin{proof}
Put $\Omega=\sigma(\Gamma)$ and $H=\mathbb T\Omega$. The scalar
factors absorb the projective multiplier, so $H$ is a subgroup of $U(d)$.
We reduce directly to a finite group by taking
\[
   H_0=H\cap SU(d).
\]
Indeed, $H$ is a finite union of scalar cosets, and each coset
$\mathbb T U$ meets $SU(d)$ in exactly $d$ points, since
\[
   \det(\zeta U)=\zeta^d\det(U).
\]
Thus $H_0$ is finite. Moreover, multiplying any element of $H$ by a
suitable scalar makes its determinant one, so $H=\mathbb T H_0$.

For $d\ge71$, Collins' quantitative Jordan theorem
\cite[Theorem~A]{Collins2007} gives an abelian subgroup $B_0\le H_0$ with
$[H_0:B_0]\le(d+1)!$. The subgroup $B=\mathbb T B_0$ is abelian and
\[
   J:=[H:B]\le[H_0:B_0]\le(d+1)!.
\]
Choose coset representatives $r_1,\ldots,r_J$ for $H/B$. Since commuting
unitaries are simultaneously diagonalizable, the algebra
\[
   \mathcal D=C^*(B)\subseteq M_d(\mathbb C)
\]
is commutative and has complex dimension at most $d$.

Every element of $\mathcal L_m(\Omega)$ belongs to one of the spaces
\[
   E_\alpha=r_{\alpha_1}\mathcal D+\cdots+r_{\alpha_m}\mathcal D,
   \qquad \alpha\in\{1,\ldots,J\}^m.
\]
There are at most $J^m$ such spaces, and each has real dimension at
most $2md$. This allows arbitrary coefficients in the linear combinations;
no bound on the individual coefficients is needed.

We recall the elementary covering estimate used here. If $E$ is a real
normed space of dimension $r$ and $B_E$ is its closed unit ball, then
$2B_E$ can be covered by at most
\[
   \left(1+\frac4\varepsilon\right)^r
\]
balls of radius $\varepsilon>0$. To prove this, choose a maximal
$\varepsilon$-separated subset of $2B_E$. Its open balls of radius
$\varepsilon/2$ are disjoint and contained in
$(2+\varepsilon/2)B_E$. Volume comparison bounds the number of selected
points, and maximality gives the required cover.

Apply this to all the $E_\alpha$, with the norm inherited from
$M_d(\mathbb C)$ and with $\varepsilon=\varepsilon_0$. It follows
that
\[
   \mathcal L_m(\Omega)\cap\{T:\|T\|\le2\}
\]
has a cover by radius-$\varepsilon_0$ balls using at most
\[
   J^m\left(1+\frac4{\varepsilon_0}\right)^{2md}
\]
centers. With $m$ fixed, the logarithm of this bound is at most
\[
   m(d+1)\log(d+1)
   +2md\log\left(1+\frac4{\varepsilon_0}\right).
\]
Dividing by $d^2$ gives a quantity tending to zero. We can therefore
choose $d_0(m)\ge71$, depending only on $m$, such that this cover uses
fewer than $2^{d^2}$ balls whenever $d\ge d_0(m)$.

If every unitary were within distance $\varepsilon_0$ of
$\mathcal L_m(\Omega)$, the approximating matrices could all be chosen
with norm less than $1+\varepsilon_0<2$. The preceding cover would then
give a cover of $U(d)$ by fewer than $2^{d^2}$ balls of radius
$2\varepsilon_0$. This contradicts
Lemma~\ref{lem:sec4-unitary-entropy} and proves the assertion.
\end{proof}

\begin{lemma}[Isotropy-degree obstruction]
\label{lem:isotropy-degree-obstruction}
Let $(z_i)$ and $(\pi_i)$ satisfy the hypotheses of
Lemma~\ref{lem:orbit-growth-obstruction}, and write
\[
   H_i=\ell^2(O_i)\otimes K_i.
\]
Let $\sigma_i\colon\Gamma_i\to U(K_i)$ be the associated projective
isotropy representations. Then the dimensions $\dim K_i$ are bounded.
\end{lemma}

\begin{proof}
Suppose that the dimensions are unbounded. By passing to a subsequence
and relabeling, we may arrange that
\[
   d_i:=\dim K_i\ge d_0(i)\qquad(i\ge1),
\]
where $d_0(i)$ is the threshold in
Lemma~\ref{lem:sec4-sparse-projective-spans} for $m=i$.
Put $\Omega_i=\sigma_i(\Gamma_i)$. Each $\Gamma_i$ is finite by
Lemma~\ref{lem:finite-orbit-form}, so we can choose a unitary
$V_i\in U(K_i)$ with
\[
   \operatorname{dist}\bigl(V_i,\mathcal L_i(\Omega_i)\bigr)
   \ge\varepsilon_0.
\]

Let $u_i\in O_i$ be the base point for induction. Identify its block
$H_{i,u_i}=\mathbb C\delta_{u_i}\otimes K_i$ with $K_i$, and define
$T_i$ to act as $V_i$ on this block and as zero on its orthogonal
complement. By \eqref{eq:sec4-prescribed-images}, there is a contraction
$x\in M$ with $\pi_i(x)=T_i$ for every $i$.

Fix $f\in\Cc(G,L)$ and decompose it into $m$ bisection-supported
sections. Let $\iota_i\colon K_i\to H_i$ be the isometry
$\iota_i\eta=\delta_{u_i}\otimes\eta$. A bisection contains at most
one arrow with source and range both equal to $u_i$. Hence the block
formula~\eqref{eq:sec4-induced-block-formula} gives
\[
   \iota_i^*\pi_i(f)\iota_i
   \in\mathcal L_m(\Omega_i).
\]
Choose any $i\ge m$. Since
$\mathcal L_m(\Omega_i)\subseteq\mathcal L_i(\Omega_i)$, we obtain
\[
\begin{aligned}
   \|x-f\|
   &\ge\|\iota_i^*\pi_i(x-f)\iota_i\|\\
   &=\|V_i-\iota_i^*\pi_i(f)\iota_i\|\\
   &\ge\varepsilon_0.
\end{aligned}
\]
The element $x$ and the positive constant $\varepsilon_0$ are independent
of $f$. This contradicts the norm density of $\Cc(G,L)$ in $M$.
\end{proof}

\subsection{Bounded finite type I degree}

\begin{proposition}[Bounded type I degree]
\label{prop:bounded-type-I-degree}
If $M=\Cr(G,\Sigma)$ is a type I von Neumann algebra, then there is
an integer $n_0\ge1$ such that every nonzero finite homogeneous central
summand of $M$ has degree at most $n_0$.
\end{proposition}

\begin{proof}
If $M=0$, take $n_0=1$. Otherwise, $M$ is finite by
Corollary~\ref{cor:no-properly-infinite-summand}. The structure theorem
for finite type I von Neumann algebras
\cite[Chapter~V]{TakesakiI} therefore gives orthogonal central
projections $(z_n)_{n\ge1}$ with
\[
   \sum_{n\ge1}^{\mathrm{sot}}z_n=1,
   \qquad z_nM\text{ homogeneous of type }I_n
   \text{ whenever }z_n\ne0.
\]

Suppose that the nonzero degrees are unbounded. Choose strictly
increasing integers $n_i$ with $z_{n_i}\ne0$. A character of
$Z(z_{n_i}M)$, together with the defining representation of
$M_{n_i}(\mathbb C)$, gives an irreducible representation
\[
   \rho_i\colon z_{n_i}M\longrightarrow M_{n_i}(\mathbb C).
\]
Extend it to $M$ by $\pi_i(a)=\rho_i(z_{n_i}a)$.
Lemma~\ref{lem:finite-orbit-form} gives
\[
   n_i=|O_i|\dim K_i.
\]
The projections $z_{n_i}$ are pairwise orthogonal, so
Lemma~\ref{lem:orbit-growth-obstruction} bounds the numbers $|O_i|$,
and Lemma~\ref{lem:isotropy-degree-obstruction} bounds the numbers
$\dim K_i$. Their products are therefore bounded, contradicting
$n_i\to\infty$.

Thus $\{n:z_n\ne0\}$ has a finite maximum $n_0$.
If $q$ is a nonzero central projection and $qM$ is homogeneous of
degree $n$, uniqueness of the homogeneous decomposition gives $q\le z_n$.
Hence $n\le n_0$, as required.
\end{proof}

\section{The finite type II obstruction}
\label{sec:type-II-obstruction}

We prove that a finite von Neumann algebra arising as a reduced twisted
Hausdorff \'etale groupoid algebra is type~I. The argument first uses
bisection-supported sections to exclude relative diffuseness in a
commutative von Neumann algebra containing the diagonal and the center.
A relative Maharam lemma then supplies a relative atom. The corresponding
corner admits a faithful family of quotients onto reduced isotropy
corners, to which Lemma~\ref{lem:vn-quotient-rigidity} applies.

Throughout this section, $G$ is a locally compact Hausdorff \'etale
groupoid, $\Sigma$ is a twist over $G$, and
\[
        M=\Cr(G,\Sigma)
\]
is a finite von Neumann algebra. Write $L\to G$ for the associated Fell
line bundle, and put
\[
        D=C_0(G^{(0)}),
        \qquad Z=Z(M).
\]
Let $T\colon M\to Z$ be the normalized faithful normal center-valued
trace. In particular,
\[
 T(1)=1,\qquad T(ab)=T(ba),\qquad T(ta)=tT(a)
 \quad(a,b\in M,\ t\in Z).
\]
We use the standard comparison theorem for finite von Neumann algebras:
\begin{equation}
 p\sim q\quad\Longleftrightarrow\quad T(p)=T(q)
 \qquad(p,q\text{ projections in }M).
 \label{eq:sec5-comparison}
\end{equation}
See, for example, \cite[Chapter~V]{TakesakiI}.

Define
\begin{equation}
        C=W^*(Z\cup D)\subseteq M.
        \label{eq:sec5-C-definition}
\end{equation}
This is a commutative von Neumann algebra, and
$E=T|_C\colon C\to Z$ is a faithful normal conditional expectation.
The von Neumann algebra generated in
\eqref{eq:sec5-C-definition} is taken inside $M$; no weak closedness of
$D$ is assumed.

For a central projection $z\in Z$, write
\[
 M_z=zM,\qquad C_z=zC,\qquad Z_z=zZ,
 \qquad T_z=T|_{M_z}.
\]
The unit of these corners is $z$.

\subsection{Relative diffuseness and the bisection estimate}

\begin{definition}[Relative trace diffuseness]
\label{def:sec5-relative-diffuse}
Let $z\in Z$ be a central projection. A nonzero projection $e\in C_z$
is \emph{$T_z$-diffuse} if, for every nonzero projection $p\in C_z$
with $p\leq e$ and every integer $n\geq2$, there are pairwise orthogonal
projections $p_1,\ldots,p_n\in C_z$ such that
\[
        p=\sum_{j=1}^n p_j,
        \qquad
        T_z(p_j)=\frac1nT_z(p)\quad(1\leq j\leq n).
\]
\end{definition}

\begin{lemma}[No relatively diffuse projection]
\label{lem:no-relative-diffuse}
For every central projection $z\in Z$, the algebra $C_z$ contains no
nonzero $T_z$-diffuse projection.
\end{lemma}

\begin{proof}
Suppose that $0\neq e\in C_z$ is $T_z$-diffuse. We construct a single
norm-one operator whose distance from every compactly supported section
is at least one.

\smallskip
\noindent\emph{Construction of the operator.}
Starting with $s_0=e$, repeatedly divide the remaining projection into
two equal-trace pieces:
\[
 s_{r-1}=b_r+s_r,
 \qquad b_rs_r=0,
 \qquad
 T_z(b_r)=T_z(s_r)=\tfrac12T_z(s_{r-1}).
\]
Faithfulness makes all these pieces nonzero. The projections $b_r$ are
pairwise orthogonal. For each $r\geq1$, divide $b_r$ into $r+1$ pieces:
\begin{equation}
 b_r=p_r+\sum_{j=1}^r q_{r,j},
 \qquad
 T_z(p_r)=T_z(q_{r,j})=\frac{T_z(b_r)}{r+1}.
 \label{eq:sec5-pr-qr-split}
\end{equation}
All these projections belong to $C_z$.

By \eqref{eq:sec5-comparison}, choose partial isometries
$v_{r,j}\in M_z$ with
\[
 v_{r,j}^*v_{r,j}=p_r,
 \qquad v_{r,j}v_{r,j}^*=q_{r,j},
\]
and set
\[
        w_r=\frac1{\sqrt r}\sum_{j=1}^r v_{r,j}.
\]
The range projections $q_{r,j}$ are orthogonal, so
$v_{r,j}^*v_{r,k}=0$ when $j\neq k$. Hence
\[
        w_r^*w_r=p_r,
        \qquad w_r=b_rw_rb_r,
        \qquad \|w_r\|=1.
\]
Since the $b_r$ are orthogonal, the series
\begin{equation}
        V=\sum_{r\geq1}^{\mathrm{sot}}w_r
        \label{eq:sec5-V-strong-sum}
\end{equation}
converges to a contraction in $M_z$. Indeed, in a faithful normal
representation, for every finite set $F\subseteq\mathbb N$,
\[
 \left\|\sum_{r\in F}w_r\xi\right\|^2
 =\sum_{r\in F}\|p_r\xi\|^2\leq\|\xi\|^2,
\]
and the same estimate for tails gives strong convergence. Moreover,
\begin{equation}
        Vp_r=w_r\quad(r\geq1),
        \qquad \|V\|=1.
        \label{eq:sec5-V-properties}
\end{equation}

\smallskip
\noindent\emph{The support bound for a compactly supported section.}
Fix $a\in\Cc(G,L)$ and decompose it as
\[
        a=a_1+\cdots+a_m,
\]
where each $a_\ell$ is supported in an open bisection. Such a finite
decomposition follows from a bisection cover of the compact support and
a partition of unity.

Each $a_\ell$ satisfies
\[
 a_\ell D a_\ell^*\subseteq D,
 \qquad a_\ell^*D a_\ell\subseteq D,
 \qquad a_\ell a_\ell^*,\ a_\ell^*a_\ell\in D.
\]
Since $a_\ell$ commutes with $Z$, these inclusions imply
\begin{equation}
        a_\ell C a_\ell^*\subseteq C,
        \qquad a_\ell^*C a_\ell\subseteq C.
        \label{eq:sec5-normalizes-C}
\end{equation}
To justify the passage to $C$, first check it on the linear span of
$Z$ and the products $td$, with $t\in Z$ and $d\in D$, and then use
ultraweak continuity of multiplication by fixed operators. This span is
the unital $*$-algebra generated by $Z\cup D$.

Fix $r$. For each $\ell$, let $f_{\ell,r}$ and $e_{\ell,r}$ be the
right and left support projections of
\[
        x_{\ell,r}=a_\ell p_r\in M_z.
\]
Thus $f_{\ell,r}$ is the support of $x_{\ell,r}^*x_{\ell,r}$ and
$e_{\ell,r}$ is the support of $x_{\ell,r}x_{\ell,r}^*$.
Both belong to $C_z$, by \eqref{eq:sec5-normalizes-C}, and
$f_{\ell,r}\leq p_r$. The polar decomposition gives
$f_{\ell,r}\sim e_{\ell,r}$, so
\[
        T_z(e_{\ell,r})=T_z(f_{\ell,r})\leq T_z(p_r).
\]
Put
\[
        e_r=\bigvee_{\ell=1}^m e_{\ell,r}\in C_z.
\]
Then
\begin{equation}
        e_r a p_r=a p_r,
        \qquad T_z(e_r)\leq mT_z(p_r).
        \label{eq:sec5-support-bound}
\end{equation}
The first identity follows from the definition of left support; the
second follows from $e_r\leq\sum_\ell e_{\ell,r}$ in the commutative
algebra $C_z$.

\smallskip
\noindent\emph{The trace estimate.}
All the projections $e_r,q_{r,1},\ldots,q_{r,r}$ belong to $C_z$ and
therefore commute. For $j\neq k$ this gives
\[
 v_{r,j}^*e_rv_{r,k}
 =v_{r,j}^*q_{r,j}e_rq_{r,k}v_{r,k}=0.
\]
Using traciality for the remaining terms, we obtain
\begin{align}
 T_z(w_r^*e_rw_r)
 &=\frac1r\sum_{j=1}^r T_z(v_{r,j}^*e_rv_{r,j})\notag\\
 &=\frac1r\sum_{j=1}^r T_z(e_rq_{r,j})\notag\\
 &\leq\frac1rT_z(e_r)
 \leq\frac mrT_z(p_r).
 \label{eq:sec5-trace-estimate}
\end{align}
Now set
\[
        Y_r=(z-e_r)(V-a)p_r.
\]
Equations \eqref{eq:sec5-V-properties} and
\eqref{eq:sec5-support-bound} give $Y_r=(z-e_r)w_r$. Consequently,
\[
 T_z(Y_r^*Y_r)
 =T_z(p_r)-T_z(w_r^*e_rw_r)
 \geq\left(1-\frac mr\right)T_z(p_r).
\]
On the other hand,
\[
        Y_r^*Y_r\leq\|V-a\|^2p_r.
\]
Combining the two inequalities yields
\[
 \|V-a\|^2T_z(p_r)
 \geq\left(1-\frac mr\right)T_z(p_r).
\]
Since $T_z(p_r)$ is positive and nonzero, this implies the scalar
inequality
\[
        \|V-a\|^2\geq1-\frac mr.
\]
For the fixed section $a$, the number $m$ is fixed and $r$ is arbitrary.
Letting $r\to\infty$ gives $\|V-a\|\geq1$. This contradicts the norm
density of $\Cc(G,L)$ in $M$.
\end{proof}

\subsection{The relative Maharam lemma}

We next show that the absence of relative atoms forces relative
diffuseness. The argument uses projections and spectral calculus, and
requires no countability assumption.

The notions of relative atom and relative atomlessness are classical
for Boolean algebras; see \cite[331A]{FremlinIII}.
We use their projection formulation for commutative von Neumann
algebras.

\begin{definition}[Relative central atom]
\label{def:sec5-relative-atom}
Let $Z_0\subseteq C_0$ be commutative von Neumann algebras with the same
unit. A nonzero projection $e\in C_0$ is a \emph{relative $Z_0$-atom} if
\[
        C_0e=Z_0e.
\]
Equivalently, every projection $f\in C_0$ with $f\leq e$ has the form
$f=we$ for a projection $w\in Z_0$.
\end{definition}

For the equivalence, if $f=te$ with $t\in Z_0$, replace $t$ by its
self-adjoint part and take $w=1_{(1/2,\infty)}(t)$; functional calculus
in the corner gives $we=f$. Conversely, suppose that every projection $p\in C_0$ with $p\leq e$
has the form $p=we$ for a projection $w\in Z_0$.
Let $a=a^*\in C_0e$ and $\varepsilon>0$. Partition the spectrum of $a$
into finitely many Borel sets $I_1,\ldots,I_n$ of diameter less than
$\varepsilon$, choose $t_j\in I_j$, and put
\[
        p_j=1_{I_j}(a),
\]
using the Borel functional calculus in the corner $C_0e$.
Then $p_j\leq e$ and
\[
        \left\|a-\sum_{j=1}^n t_jp_j\right\|<\varepsilon.
\]
By hypothesis, $p_j=w_je$ for projections $w_j\in Z_0$, so
\[
        \sum_{j=1}^n t_jp_j
        =\left(\sum_{j=1}^n t_jw_j\right)e
        \in Z_0e.
\]
The algebra $Z_0e$ is norm closed, being the image of the
$*$-homomorphism $Z_0\to C_0e$, $t\mapsto te$.
Since $\varepsilon$ is arbitrary, $a\in Z_0e$.
Applying this to the real and imaginary parts of an arbitrary element
of $C_0e$ gives $C_0e\subseteq Z_0e$.
The reverse inclusion is immediate.

The following is a von Neumann algebraic formulation of Maharam's
lemma; see \cite[331B]{FremlinIII}. For a formulation in terms of
conditional expectations on a probability space, see
\cite[Lemma~2]{Delbaen2020}.

We include a direct proof, adapting the small-piece and maximality
argument to projections and a faithful normal conditional expectation.
This establishes the statement for arbitrary commutative von Neumann
algebras, without assuming the existence of a faithful normal state.

\begin{lemma}[Relative Maharam lemma]
\label{lem:relative-maharam}
Let $Z_0\subseteq C_0$ be commutative von Neumann algebras with the same
unit, and let $E\colon C_0\to Z_0$ be a faithful normal conditional
expectation. Suppose that a projection $p\in C_0$ contains no nonzero
relative $Z_0$-atom. For every $h\in(Z_0)_+$ with
\[
        0\leq h\leq E(p),
\]
there is a projection $q\leq p$ such that $E(q)=h$.

In particular, for every nonzero projection $s\leq p$ and every
integer $n\geq2$, there are pairwise orthogonal projections
$s_1,\ldots,s_n\in C_0$ such that
\[
        s=\sum_{j=1}^n s_j,
        \qquad
        E(s_j)=\frac1nE(s)\quad(1\leq j\leq n).
\]
\end{lemma}

\begin{proof}
\emph{Small subprojections.}
Let $0\neq s\leq p$ be a projection. Since $s$ is not a relative atom,
there is a projection $f\leq s$ which is not $ws$ for any projection
$w\in Z_0$. Put
\[
 a=E(f),\qquad b=E(s-f),\qquad
 w=1_{(-\infty,0]}(a-b).
\]
The element
\[
        t=wf+(1-w)(s-f)
\]
is a projection below $s$, and bimodularity of $E$ gives
\[
 E(t)=wa+(1-w)b\leq\tfrac12(a+b)=\tfrac12E(s).
\]
Moreover, $t\neq0$: otherwise $wf=0$ and $(1-w)(s-f)=0$, which would
give $f=(1-w)s$, contrary to the choice of $f$.

Every nonzero subprojection of $p$ is again not a relative atom.
Iteration therefore shows that, for every integer $k\geq1$ and every
nonzero projection $s\leq p$, there is a nonzero projection $t\leq s$
such that
\begin{equation}
        E(t)\leq2^{-k}E(s).
        \label{eq:sec5-small-piece}
\end{equation}

\smallskip
\noindent\emph{Prescribing the expectation.}
Fix $0\leq h\leq E(p)$, and order
\[
 \mathcal S_h=
 \{q\in\operatorname{Proj}(C_0):q\leq p,\ E(q)\leq h\}
\]
by inclusion. This set contains $0$. Every chain has an upper bound in
$\mathcal S_h$: take its supremum in the projection lattice and use
normality of $E$. Zorn's lemma gives a maximal element $q$.

Suppose that $\delta=h-E(q)$ is nonzero, and put
\[
        r=p-q,\qquad b=E(r).
\]
Then $0\leq\delta\leq b$. For some integer $k\geq1$, the spectral
projection
\[
        w=1_{(0,\infty)}(\delta-2^{-k}b)
\]
is nonzero. Indeed, if all these projections were zero, then
$\delta\leq2^{-k}b$ for every $k$, forcing $\delta=0$.
On this projection,
\[
        w\delta\geq2^{-k}wb.
\]
Also $wb\neq0$: since $0\leq\delta\leq b$, the positive spectral
projection $w$ is supported on the support of $b$. Hence
$E(wr)=wb\neq0$, and in particular $wr\neq0$.

Apply \eqref{eq:sec5-small-piece} to $wr$ with this same $k$. We obtain
a nonzero projection $t\leq wr$ satisfying
\[
        E(t)\leq2^{-k}E(wr)=2^{-k}wb\leq w\delta\leq\delta.
\]
Since $t\leq p-q$, the projection $q+t$ strictly contains $q$, while
\[
        E(q+t)\leq E(q)+\delta=h.
\]
This contradicts maximality. Therefore $E(q)=h$.

Finally, let $0\neq s\leq p$ and $n\geq2$. Apply the result successively
to choose $n-1$ orthogonal subprojections of $s$, each with expectation
$E(s)/n$, and take the remaining projection as the last piece.
At each stage the remainder has no relative atom and its expectation
dominates $E(s)/n$. The last piece also has expectation $E(s)/n$.
Faithfulness ensures that all pieces are nonzero.
\end{proof}

\begin{corollary}[Relative atoms in every nonzero corner]
\label{cor:sec5-relative-atoms}
Let $z\in Z$ be a nonzero central projection. Every nonzero projection
$p\in C_z$ contains a nonzero relative $Z_z$-atom.
\end{corollary}

\begin{proof}
Apply Lemma~\ref{lem:relative-maharam} to the expectation
$T_z|_{C_z}\colon C_z\to Z_z$. If $p$ contained no relative atom, it
would be $T_z$-diffuse, contrary to Lemma~\ref{lem:no-relative-diffuse}.
\end{proof}

\subsection{Reduced isotropy quotients from a relative atom}

The next step connects the relative atom to the groupoid. The identity
$Ce=Ze$ makes the corner $eMe$ commute with $D$. Its regular
representations therefore preserve the subspaces on which $D$ acts by
evaluation at a single unit. These subspaces identify with isotropy
regular representations.

\begin{lemma}[Faithful family of reduced isotropy-corner quotients]
\label{lem:relative-atom-isotropy-corners}
Let $0\neq e\in C$ be a relative $Z$-atom, and put $N=eMe$.
For every $u\in G^{(0)}$ and $x\in G\cdot u$, there are a projection
\[
        q_{u,x}\in\Cr(G_x^x,\Sigma|_{G_x^x})
\]
and a surjective $*$-homomorphism
\begin{equation}
 \rho_{u,x}\colon N\longrightarrow
 q_{u,x}\Cr(G_x^x,\Sigma|_{G_x^x})q_{u,x}.
 \label{eq:sec5-rho-map}
\end{equation}
Some of these projections may be zero, but the family is faithful:
\[
        \bigcap_{u\in G^{(0)}}\ \bigcap_{x\in G\cdot u}
        \ker\rho_{u,x}=\{0\}.
\]
\end{lemma}

\begin{proof}
\emph{Commutation with the diagonal.}
For $d\in D$, the identity $Ce=Ze$ gives $de=t_de$ for some $t_d\in Z$.
Since $e$ commutes with $d$, every $b=eae\in N$ satisfies
\[
        db=t_db=bd.
\]
Thus $N\subseteq D'\cap M$.

Fix $u\in G^{(0)}$. Decompose the source-regular Hilbert space by range:
\begin{equation}
 \mathcal H_u=\ell^2(G_u,L)
 =\bigoplus_{x\in G\cdot u}\mathcal H_{u,x},
 \qquad \mathcal H_{u,x}=\ell^2(G_u^x,L),
 \label{eq:sec5-range-blocks}
\end{equation}
where $G_u^x=\{\gamma\in G:s(\gamma)=u,\ r(\gamma)=x\}$.
Let $P_{u,x}$ be the projection onto $\mathcal H_{u,x}$.
On this subspace, $\lambda_u(d)$ acts as $d(x)I$.
For $x\neq y$, choose $d\in D$ with $d(x)=0$ and $d(y)=1$.
For $b\in N$, commutation with $d$ gives
\[
 P_{u,y}\lambda_u(b)P_{u,x}
 =P_{u,y}\lambda_u(d)\lambda_u(b)P_{u,x}
 =P_{u,y}\lambda_u(b)\lambda_u(d)P_{u,x}=0.
\]
It follows that every range block reduces $\lambda_u(N)$.

\smallskip
\noindent\emph{Identification of the compressed operators.}
Fix $x\in G\cdot u$, choose an arrow $\eta\colon u\to x$, and put
$\Gamma_x=G_x^x$. The map $g\mapsto g\eta$ is a bijection from
$\Gamma_x$ onto $G_u^x$. Choose a unit vector $v\in L_\eta$.
Multiplication by $v$ gives a unitary
\[
 W_{u,x}\colon\ell^2(\Gamma_x,L|_{\Gamma_x})
       \longrightarrow\mathcal H_{u,x},
 \qquad (W_{u,x}\xi)_{g\eta}=\xi_gv.
\]
For $f\in\Cc(G,L)$, the regular representation formula gives
\begin{equation}
 W_{u,x}^*P_{u,x}\lambda_u(f)P_{u,x}W_{u,x}
 =\lambda_{\Gamma_x}(f|_{\Gamma_x}).
 \label{eq:sec5-block-is-isotropy}
\end{equation}
Indeed, the coefficient indexed by $g\eta,h\eta$ uses
\[
        (g\eta)(h\eta)^{-1}=gh^{-1},
\]
and the line-bundle transports by $v$ give precisely the isotropy
convolution formula. The restriction $f|_{\Gamma_x}$ has finite support
because $\Gamma_x$ is closed and discrete in $G$.

Compression is contractive. Identifying the isotropy algebra with its
faithful regular image, \eqref{eq:sec5-block-is-isotropy} therefore
extends to a bounded linear map
\[
 R_{u,x}\colon M\longrightarrow\Cr(\Gamma_x,\Sigma|_{\Gamma_x}).
\]
Its range contains $\Cc(\Gamma_x,L|_{\Gamma_x})$. To see this, extend
each of the finitely many values of an isotropy section using a
trivializing open bisection and a compactly supported cutoff function.
Each such bisection meets $\Gamma_x$ in at most one arrow. Summing these
extensions gives the required section of $\Cc(G,L)$.

\smallskip
\noindent\emph{The corner maps and their faithfulness.}
Set $q_{u,x}=R_{u,x}(e)$. Since the range block reduces
$\lambda_u(N)$ and $e$ is the unit of $N$, this is a projection and
\[
        \rho_{u,x}:=R_{u,x}|_N
\]
is a $*$-homomorphism into the indicated isotropy corner. Moreover,
$P_{u,x}$ commutes with $\lambda_u(e)$, so for every $a\in M$,
\begin{equation}
        \rho_{u,x}(eae)=q_{u,x}R_{u,x}(a)q_{u,x}.
        \label{eq:sec5-corner-compression}
\end{equation}
The extension argument above shows that the range of $\rho_{u,x}$
contains
\[
        q_{u,x}\Cc(\Gamma_x,L|_{\Gamma_x})q_{u,x},
\]
which is dense in $q_{u,x}\Cr(\Gamma_x,\Sigma|_{\Gamma_x})q_{u,x}$.
The range of a $*$-homomorphism between $C^*$-algebras is norm closed,
so $\rho_{u,x}$ is surjective.

Finally, if $b\in N$ belongs to every kernel, then every range block of
$\lambda_u(b)$ vanishes, for every $u$. By
\eqref{eq:sec5-range-blocks} and block diagonality, $\lambda_u(b)=0$
for all $u$. Faithfulness of the family of source-regular
representations implies $b=0$.
\end{proof}

\begin{remark}
The maps $R_{u,x}$ need not be multiplicative on $M$; multiplicativity
holds on $N$ because its operators preserve the range blocks.
All these assertions follow from algebraic identities and norm
approximation. No normality of the regular representations, and no
exactness of a restriction sequence, is required.
\end{remark}

\subsection{Finite-dimensional representations of type II algebras}

\begin{lemma}
\label{lem:type-II-no-finite-representations}
Let $N$ be a nonzero finite type~II von Neumann algebra. Every
$*$-homomorphism from $N$ to a finite-dimensional $C^*$-algebra is zero.
\end{lemma}

\begin{proof}
A finite-dimensional $C^*$-algebra embeds in some $M_n(\mathbb C)$,
so suppose that $\theta\colon N\to M_n(\mathbb C)$ is nonzero.
Then $\theta(1_N)$ is a nonzero projection.

The continuous-dimension theorem for finite type~II von Neumann
algebras gives a decomposition
\begin{equation}
        1_N=p_1+\cdots+p_{n+1}
        \label{eq:sec5-typeII-equal-projections}
\end{equation}
into mutually equivalent orthogonal projections; see
\cite[Chapter~V]{TakesakiI}. Equivalently, the pieces may be chosen
with normalized center-valued trace $1_N/(n+1)$, and comparison gives
their equivalence.

The projections $\theta(p_j)$ are orthogonal and mutually equivalent.
If one were zero, all would be zero, contradicting
$\theta(1_N)\neq0$. Thus $M_n(\mathbb C)$ would contain $n+1$ nonzero
orthogonal projections, which is impossible. This uses only the finite
algebraic decomposition \eqref{eq:sec5-typeII-equal-projections};
normality of $\theta$ is not needed.
\end{proof}

\subsection{Excluding the type II central summand}

\begin{proposition}[Finite groupoid von Neumann algebras are type I]
\label{prop:no-finite-type-II-summand}
If $M=\Cr(G,\Sigma)$ is a finite von Neumann algebra, then $M$ is type~I.
\end{proposition}

\begin{proof}
Suppose that there is a nonzero central projection $z\in Z$ such that
$zM$ is type~II. By Corollary~\ref{cor:sec5-relative-atoms}, applied to
$p=z$, there is a nonzero relative $Z_z$-atom $e\in C_z$.
Since $e\leq z$,
\[
        Ce=C_ze=Z_ze=Ze.
\]

The corner $N=eMe$ is a nonzero finite type~II von Neumann algebra.
Indeed, finiteness passes to corners. If $N$ had a nonzero type~I
central summand, it would contain a nonzero abelian projection
$p\leq e$; see \cite[Chapter~V]{TakesakiI}. But then
\[
        p(zM)p=pMp=pNp
\]
would be abelian, contradicting the fact that the type~II algebra
$zM$ has no nonzero abelian projections.

Lemma~\ref{lem:relative-atom-isotropy-corners} supplies a faithful
family of surjective maps from $N$ onto reduced isotropy corners.
Since $N\neq0$, at least one of these maps is nonzero. Thus, for some
$x\in G^{(0)}$ and some nonzero projection $q$, there is a surjective
$*$-homomorphism
\[
        \rho\colon N\longrightarrow
        q\Cr(G_x^x,\Sigma|_{G_x^x})q\neq0.
\]

Choose unit vectors in the fibres of $L|_{G_x^x}$, taking the canonical
unit vector at the identity. Their multiplication determines a
normalized scalar cocycle $\omega_x$, identifying the target with
\[
        q\Cr(G_x^x,\omega_x)q.
\]
Lemma~\ref{lem:vn-quotient-rigidity}, applied to $\rho$ with matrix
size $k=1$, shows that this corner is finite-dimensional. This
contradicts Lemma~\ref{lem:type-II-no-finite-representations}, because
$N$ is type~II and $\rho\neq0$.

Therefore the type~II central part of $M$ is zero. Since $M$ is finite,
only its finite type~I part remains.
\end{proof}

\section{Proof of the classification theorem}
\label{sec:classification}

We now combine the preceding results to prove
Theorem~\ref{thm:main}. Recall that a $C^*$-algebra is
\emph{subhomogeneous of degree at most $N$} if every irreducible
representation has dimension at most $N$. These representations are
not assumed to be normal when the algebra is a von Neumann algebra.

We first establish the equivalence of the two algebraic descriptions.
We then construct the groupoid model and prove the converse using the
obstructions obtained in the preceding sections.

\begin{proof}[Proof of Theorem~\ref{thm:main}]
If $M=0$, take $N=1$, $A_1=0$, and the empty groupoid with its trivial
twist. Henceforth assume that $M\neq0$.

\smallskip
\noindent\emph{\textup{(2)} $\Longrightarrow$ \textup{(3)}.}
Suppose that every irreducible representation of $M$ has dimension at
most $N$.

We first observe that $M$ cannot contain a copy of $M_{N+1}(\C)$,
even as a nonunital subalgebra. Indeed, suppose that
$(e_{ij})_{i,j=1}^{N+1}$ were a nonzero system of matrix units in $M$.
Since irreducible representations of a $C^*$-algebra separate points,
there would be an irreducible representation $\pi$ of $M$ with
$\pi(e_{11})\neq0$. The matrix-unit identities would then make
\[
        \pi(e_{11}),\ldots,\pi(e_{N+1,N+1})
\]
nonzero mutually orthogonal projections: the partial isometries
$\pi(e_{j1})$ implement their equivalence. Thus the representation
space of $\pi$ would have dimension at least $N+1$, a contradiction.

It follows that $M$ is finite. Otherwise,
Theorem~\ref{thm:nonfinite-uncontrolled} would give a copy of
$B(\ell^2(\mathbb N))$ in $M$, and hence a copy of $M_{N+1}(\C)$.

The type~II central part of $M$ is also zero. To see this, let
$z_{\mathrm{II}}\in Z(M)$ be the central projection supporting that
part. For every irreducible representation $\pi$ of $M$, the
restriction
\[
        \pi|_{z_{\mathrm{II}}M}
\]
is a representation into a finite-dimensional matrix algebra.
Lemma~\ref{lem:type-II-no-finite-representations} makes this
restriction zero. Hence $\pi(z_{\mathrm{II}})=0$ for every
irreducible $\pi$, and separation of points gives
$z_{\mathrm{II}}=0$.

Thus $M$ is a finite type~I von Neumann algebra. Its homogeneous
central decomposition provides mutually orthogonal central
projections $(z_n)_{n\geq1}$ such that
\[
 \sum_{n\geq1}^{\mathrm{sot}}z_n=1,
 \qquad
 z_nM\cong Z(z_nM)\mathbin{\bar\otimes}M_n(\C)
 \quad\text{when }z_n\neq0;
\]
see \cite[Chapter~V]{TakesakiI}.
If $n>N$ and $z_n\neq0$, the constant matrix factor in this summand
would contain a nonunital copy of $M_{N+1}(\C)$, contrary to the
observation above. Therefore $z_n=0$ for $n>N$.

Set $A_n=Z(z_nM)$ when $z_n\neq0$, and set $A_n=0$ otherwise.
Only finitely many summands remain, so
\begin{equation}
        M\cong
        \prod_{n=1}^N
        \bigl(A_n\mathbin{\bar\otimes}M_n(\C)\bigr).
        \label{eq:sec6-subhomogeneous-structure}
\end{equation}
This is condition \textup{(3)}. Notice that this argument used only
general facts about von Neumann algebras, without a groupoid
realization.

\smallskip
\noindent\emph{\textup{(3)} $\Longrightarrow$ \textup{(2)}.}
Suppose that $M$ has the product decomposition
\eqref{eq:sec6-subhomogeneous-structure}.
In an irreducible representation of a finite product, the central
coordinate projections act as either zero or the identity.
Exactly one acts as the identity, because their sum is $1$.
Thus every irreducible representation factors through one nonzero
coordinate
\[
        A_n\mathbin{\bar\otimes}M_n(\C).
\]
By Lemma~\ref{lem:sec4-constant-matrix-lift}, every irreducible
representation of this homogeneous summand has dimension $n$.
Consequently, every irreducible representation of $M$ has dimension
at most $N$, proving \textup{(2)}.

\smallskip
\noindent\emph{\textup{(3)} $\Longrightarrow$ \textup{(1)}.}
We construct the groupoid explicitly. Discard the zero factors in
\eqref{eq:sec6-subhomogeneous-structure}. For each remaining $n$,
let $X_n$ be the Gelfand spectrum of the unital abelian
$C^*$-algebra $A_n$. Then $X_n$ is compact Hausdorff and
\[
        A_n\cong C(X_n)
\]
as $C^*$-algebras.

Let $I_n=\{1,\ldots,n\}$ with the discrete topology, and define
\begin{equation}
        G_n=X_n\times I_n\times I_n.
        \label{eq:sec6-pair-groupoid}
\end{equation}
Identify its unit space with $X_n\times I_n$ through
$(x,i)\leftrightarrow(x,i,i)$, and use the structure maps
\begin{align*}
 s(x,i,j)&=(x,j),
 & r(x,i,j)&=(x,i),\\
 (x,i,j)(x,j,k)&=(x,i,k),
 & (x,i,j)^{-1}&=(x,j,i).
\end{align*}
Thus $G_n$ is the pair groupoid on $I_n$ over each point of $X_n$.

The groupoid $G_n$ is compact and Hausdorff. It is \'etale because,
for each open $U\subseteq X_n$, the set
\[
        U\times\{i\}\times\{j\}
\]
is an open bisection. It is principal because its only isotropy
arrows are the units $(x,i,i)$. Equip it with the trivial twist
$G_n\times\mathbb T$.

Since $G_n$ is compact, its convolution algebra has underlying
vector space $C(G_n)$. Define
\begin{equation}
 \Phi_n\colon C_c(G_n)\longrightarrow C(X_n,M_n(\C)),
 \qquad
 \Phi_n(f)(x)=\bigl[f(x,i,j)\bigr]_{i,j=1}^n.
 \label{eq:sec6-Phi-pair}
\end{equation}
The formulas
\begin{align*}
 (f*g)(x,i,k)
     &=\sum_{j=1}^n f(x,i,j)g(x,j,k),\\
 f^*(x,i,j)
     &=\overline{f(x,j,i)}
\end{align*}
show that $\Phi_n$ is an algebraic $*$-isomorphism.

To compute the reduced norm, fix a unit $(x,j)$. Its source fibre is
\[
        (G_n)_{(x,j)}
        =\{(x,i,j):1\leq i\leq n\}.
\]
Under the corresponding identification of its regular Hilbert
space with $\C^n$, the source-regular representation of $f$ is
exactly the matrix $\Phi_n(f)(x)$. Therefore
\begin{equation}
        \|f\|_{C_r^*(G_n)}
        =\sup_{x\in X_n}\|\Phi_n(f)(x)\|.
        \label{eq:sec6-pair-reduced-norm}
\end{equation}
It follows that $\Phi_n$ extends to an isomorphism
\begin{equation}
 \begin{split}
 C_r^*(G_n)
 &\cong C(X_n,M_n(\C))\\
 &\cong M_n(A_n)
 \cong A_n\mathbin{\bar\otimes}M_n(\C).
 \end{split}
 \label{eq:sec6-pair-algebra}
\end{equation}
Here the $C^*$-algebra and von Neumann algebra matrix
amplifications coincide: both consist of $n\times n$ matrices
with entries in $A_n$.

Finally, take the finite disjoint union
\begin{equation}
        G=
        \bigsqcup_{\substack{1\leq n\leq N\\A_n\neq0}}G_n
        \label{eq:sec6-disjoint-union}
\end{equation}
and the trivial twist $\Sigma=G\times\mathbb T$.
The groupoid $G$ is compact, Hausdorff, \'etale, and principal.
Convolution acts separately on its components, and the reduced
norm is the maximum of their reduced norms. Hence
\[
 \begin{split}
 C_r^*(G,\Sigma)
 &\cong
 \bigoplus_{\substack{1\leq n\leq N\\A_n\neq0}}C_r^*(G_n)\\
 &\cong
 \prod_{n=1}^N
 \bigl(A_n\mathbin{\bar\otimes}M_n(\C)\bigr)
 \cong M.
 \end{split}
\]
The direct sum and direct product agree because the index set is
finite. This proves \textup{(1)}, including the additional assertion
about a compact principal groupoid and a trivial twist.

\smallskip
\noindent\emph{\textup{(1)} $\Longrightarrow$ \textup{(3)}.}
Suppose that $M$ has a groupoid realization as in \textup{(1)}.
Identify
\[
        M=\Cr(G,\Sigma).
\]
Corollary~\ref{cor:no-properly-infinite-summand} gives directly
\begin{equation}
        M\text{ is finite}.
        \label{eq:sec6-M-finite}
\end{equation}
Proposition~\ref{prop:no-finite-type-II-summand} then shows that
$M$ is type~I. Finally,
Proposition~\ref{prop:bounded-type-I-degree} provides an integer
$N\geq1$ such that every nonzero finite homogeneous central
summand of $M$ has degree at most $N$.

Since $M$ is finite, its homogeneous type~I decomposition contains
only finite degrees. The preceding bound therefore leaves only
the summands of types $I_1,\ldots,I_N$. Applying the homogeneous
structure theorem once more gives
\[
        M\cong
        \prod_{n=1}^N
        \bigl(A_n\mathbin{\bar\otimes}M_n(\C)\bigr),
\]
where $A_n=Z(z_nM)$ for each nonzero summand and $A_n=0$
otherwise. This proves \textup{(3)}.

Together with the other implications, this establishes the
equivalence of all three conditions.
\end{proof}

\appendix

\section{Unitality and the Haar-system setting}
\label{sec:unitality-haar}

In this appendix, $G$ is a locally compact Hausdorff groupoid
with a continuous Haar system $\lambda=(\lambda^u)_{u\in G^{(0)}}$
of full support. Let $\Sigma$ be a twist over $G$, with associated
Fell line bundle $L\to G$. The notation $C^*(G,\Sigma;\lambda)$ and
$C_r^*(G,\Sigma;\lambda)$ records the Haar system used in the convolution
and the completions. No countability assumption is imposed.

Let $\lambda_u$ be the image of $\lambda^u$ under inversion, regarded
as a measure on $G_u=s^{-1}(u)$. The source-regular representation
acts on $\mathcal H_u=L^2(G_u,L,\lambda_u)$ by
\[
 (\pi_u(f)\xi)(\gamma)
 =\int_{G_u}f(\gamma\eta^{-1})\xi(\eta)\,d\lambda_u(\eta).
\]
These representations are nondegenerate and define the reduced norm.

\begin{theorem}[Unitality criterion]
\label{thm:unitality-haar}
The following conditions are equivalent:
\begin{enumerate}
\item $C_r^*(G,\Sigma;\lambda)$ is unital;
\item $C^*(G,\Sigma;\lambda)$ is unital;
\item $G$ is \'etale and $G^{(0)}$ is compact.
\end{enumerate}
\end{theorem}

\begin{proof}
Assume first that the reduced algebra is unital. Set $\varepsilon=1/4$
and choose $f\in C_c(G,L)$ with $\|f-1\|_r<\varepsilon$.
We use this one section throughout the argument.

\smallskip
\noindent\emph{The source fibres are discrete.}
Fix $u\in G^{(0)}$, and let $U\subseteq G_u$ be relatively compact
and open. Let $P_U$ be the projection onto
$\mathcal H_U=L^2(U,L,\lambda_u)$. The operator
$K_U=\left.P_U\pi_u(f)\right|_{\mathcal H_U}$ has kernel
$f(\gamma\eta^{-1})$, viewed as a map $L_\eta\to L_\gamma$,
for $\gamma,\eta\in U$. The measure $\lambda_u(U)$ is finite, and
\[
 \int_U\int_U\|f(\gamma\eta^{-1})\|^2
       \,d\lambda_u(\eta)\,d\lambda_u(\gamma)
 \leq \|f\|_\infty^2\lambda_u(U)^2<\infty.
\]
Using finitely many trivializations of $L$ over $\overline U$, this is
a square-integrable line-bundle kernel. Hence $K_U$ is Hilbert--Schmidt
and, in particular, compact.

Nondegeneracy of $\pi_u$ gives
\[
        \|K_U-I_{\mathcal H_U}\|<\varepsilon<1.
\]
Thus $K_U$ is invertible, so $\mathcal H_U$ is finite-dimensional.
Full support now forces $U$ to be finite. Indeed, if $U$ were infinite,
one could choose arbitrarily many disjoint nonempty open subsets of $U$
and nonzero square-integrable sections supported in them, producing
arbitrarily many orthogonal vectors in $\mathcal H_U$.
Every point of $G_u$ has a finite open neighbourhood. Since $G_u$ is
Hausdorff, it follows that $G_u$ is discrete.

\smallskip
\noindent\emph{The unit space is compact and open.}
Put $m(u)=\lambda_u(\{u\})>0$, where positivity follows from
discreteness and full support. Right invariance gives
\[
        \lambda_{s(\gamma)}(\{\gamma\})=m(r(\gamma)).
\]
Using normalized delta vectors in the regular representations, and
the canonical identification $L_u=\mathbb C$ at units, we obtain
\begin{align}
 |m(u)f(u)-1|&\leq\varepsilon
       &&(u\in G^{(0)}),\label{eq:unital-unit-coefficient}\\
 \sqrt{m(s(\gamma))m(r(\gamma))}\,\|f(\gamma)\|
       &\leq\varepsilon
       &&(\gamma\notin G^{(0)}).
       \label{eq:unital-off-unit-coefficient}
\end{align}
For the second inequality, the relevant coefficient is between
the delta vectors at $s(\gamma)$ and $\gamma$; these are orthogonal
when $\gamma$ is not a unit.

Equation~\eqref{eq:unital-unit-coefficient} shows that $f(u)\neq0$
for every unit. Therefore $G^{(0)}\subseteq\operatorname{supp}(f)$.
The unit space is closed in the Hausdorff groupoid $G$, so it is compact.

Define the continuous function
\[
 h(\gamma)=
 \frac{\|f(\gamma)\|}
 {\sqrt{|f(s(\gamma))|\,|f(r(\gamma))|}}.
\]
The denominator is everywhere positive, and $h(u)=1$ at every unit.
Moreover, \eqref{eq:unital-unit-coefficient} gives
$m(u)|f(u)|\geq1-\varepsilon$. Combining this with
\eqref{eq:unital-off-unit-coefficient}, we obtain
\[
        h(\gamma)\leq\frac{\varepsilon}{1-\varepsilon}
        =\frac13
        \qquad(\gamma\notin G^{(0)}).
\]
Consequently,
\[
        G^{(0)}=\{\gamma\in G:h(\gamma)>1/2\}
\]
is open.

A continuous Haar system of full support makes $r$ and $s$ open.
For example, for an open set $V\subseteq G$, the positivity sets of
the continuous functions
\[
        u\longmapsto\int_{G^u}\varphi\,d\lambda^u,
        \qquad
        0\leq\varphi\in C_c(G),\quad
        \operatorname{supp}(\varphi)\subseteq V,
\]
cover $r(V)$ and are contained in it; inversion gives the assertion
for $s$.
Now fix $\gamma\in G$. Since the unit space is open, continuity of
multiplication and inversion gives an open neighbourhood $V$ of
$\gamma$ such that
\[
        V^{-1}V\subseteq G^{(0)},
        \qquad
        VV^{-1}\subseteq G^{(0)}.
\]
These inclusions make both $r|_V$ and $s|_V$ injective.
Since they are open, $V$ is a bisection. Thus $G$ is \'etale,
proving \textup{(1)}$\Rightarrow$\textup{(3)}.

\smallskip
The regular quotient
$C^*(G,\Sigma;\lambda)\twoheadrightarrow C_r^*(G,\Sigma;\lambda)$
shows that \textup{(2)} implies \textup{(1)}.

Finally, assume \textup{(3)}. The function
$m(u)=\lambda_u(\{u\})$ is positive and continuous.
Indeed, Haar continuity applied to a function
$d\in C_c(G^{(0)})$, extended by zero to $G$, gives continuity of
$u\mapsto d(u)m(u)$; choosing $d=1$ near any prescribed unit proves
the assertion locally.
Since $G^{(0)}$ is compact and open, the section
\[
 e_\lambda(\gamma)=
 \begin{cases}
   m(u)^{-1}1_u,&\gamma=u\in G^{(0)},\\
   0,&\gamma\notin G^{(0)}
 \end{cases}
\]
belongs to $C_c(G,L)$. Here $1_u$ is the canonical unit vector in $L_u$.
Right invariance and the convolution formula give
$e_\lambda*a=a*e_\lambda=a$ for every $a\in C_c(G,L)$.
Thus both completions are unital, proving the remaining implications.
\end{proof}

\begin{remark}[Normalization of the Haar system]
\label{rem:haar-counting-normalization}
Under the equivalent conditions of
Theorem~\ref{thm:unitality-haar}, the given Haar system can be replaced
by counting measures without changing the reduced algebra up to
isomorphism. Explicitly, the map
\[
        (\Psi a)(\gamma)
        =\sqrt{m(r(\gamma))m(s(\gamma))}\,a(\gamma)
\]
is a $*$-isomorphism from the convolution algebra defined using
$\lambda$ onto the convolution algebra defined using counting measures.
For each unit $u$, the unitary
\[
 \mathcal U_u:L^2(G_u,L,\lambda_u)\longrightarrow\ell^2(G_u,L),
 \qquad
 (\mathcal U_u\xi)_\gamma=\sqrt{m(r(\gamma))}\,\xi_\gamma
\]
satisfies
\[
        \mathcal U_u\pi_u(a)\mathcal U_u^*
        =\lambda_u^{\mathrm{count}}(\Psi a).
\]
Hence $\Psi$ extends to an isomorphism
\[
        C_r^*(G,\Sigma;\lambda)\cong C_r^*(G,\Sigma),
\]
where the algebra on the right uses counting measures, as in the main
text. In particular, the section $e_\lambda$ is carried to the
canonical unit section on $G^{(0)}$.
\end{remark}

\begin{corollary}[Classification in the Haar-system setting]
\label{cor:haar-system-classification}
Let $M$ be a nonzero von Neumann algebra. Then $M$ is isomorphic,
as a $C^*$-algebra, to $C_r^*(G,\Sigma;\lambda)$ for a locally compact
Hausdorff groupoid with a continuous Haar system of full support and
a twist $\Sigma$ if and only if $M$ is subhomogeneous.

When these conditions hold, the groupoid may be chosen compact and
principal, the twist trivial, and the Haar system the counting system.
\end{corollary}

\begin{proof}
If $M\cong C_r^*(G,\Sigma;\lambda)$, the reduced algebra is unital.
Theorem~\ref{thm:unitality-haar} forces $G$ to be \'etale with compact
unit space. Remark~\ref{rem:haar-counting-normalization} identifies this
algebra with the reduced algebra for counting measures, so
Theorem~\ref{thm:main} implies that $M$ is subhomogeneous.
Conversely, the compact principal groupoid model in
Theorem~\ref{thm:main}, with its trivial twist and counting Haar system,
gives the required realization.
\end{proof}

\bibliographystyle{amsalpha}

\providecommand{\bysame}{\leavevmode\hbox to3em{\hrulefill}\thinspace}
\providecommand{\MR}{\relax\ifhmode\unskip\space\fi MR }
\providecommand{\MRhref}[2]{%
  \href{http://www.ams.org/mathscinet-getitem?mr=#1}{#2}
}
\providecommand{\href}[2]{#2}

\end{document}